\documentclass[a4paper,10pt,reqno]{amsart} 

\usepackage[english]{babel}
\usepackage[utf8]{inputenc}

\usepackage{amssymb}
\usepackage{amsmath}
\usepackage{amsthm}
\usepackage{bbm}
\usepackage{mathtools}

\usepackage{enumitem}
\usepackage{tikz}
\usetikzlibrary{matrix}
\usetikzlibrary{hobby}
\usepackage{marginnote}
\usepackage{hyperref}

\newcommand{\bbC}{\mathbb{C}}

\newcommand{\bbN}{\mathbb{N}}

\newcommand{\bbR}{\mathbb{R}}

\newcommand{\calB}{\mathcal{B}}

\newcommand{\calH}{\mathcal{H}}

\newcommand{\suchthat}{\,\middle|\,} 
\newcommand{\suchthatManual}{\, | \,}
\DeclareMathOperator{\id}{id} 
\DeclareMathOperator{\one}{\mathbbm{1}} 
\DeclareMathOperator{\re}{Re} 
\newcommand{\ui}{\mathrm{i}} 
\newcommand{\argument}{\mathord{\,\cdot\,}} 
\newcommand{\dx}{\;\mathrm{d}} 
\newcommand{\dxInt}{\;\mathrm{d}} 
\newcommand{\norm}[1]{\left\lVert #1 \right\rVert} 
\newcommand{\inner}[2]{\left( #1 \middle| #2 \right)} 
\newcommand{\modulus}[1]{\left\lvert #1 \right\rvert} 
\DeclareMathOperator{\dom}{dom} 

\newcommand{\numRange}{\operatorname{W}} 
\newcommand{\spec}{\sigma} 
\newcommand{\specPnt}{\spec_{\operatorname{pnt}}} 
\newcommand{\spb}{s} 

\newcommand{\impliesProof}[2]{``\ref{#1} $\Rightarrow$ \ref{#2}''}
\newcommand{\equivalentProof}[2]{``\ref{#1} $\Leftrightarrow$ \ref{#2}''}

\newcommand{\ext}{{\operatorname{ext}}}
\newcommand{\bdd}{{\operatorname{B}}}

\theoremstyle{definition}
\newtheorem{definition}{Definition}[section]
\newtheorem{remark}[definition]{Remark}

\newtheorem{example}[definition]{Example}

\newtheorem{splitting_assumption}[definition]{Splitting Assumption}

\theoremstyle{plain}
\newtheorem{proposition}[definition]{Proposition}
\newtheorem{lemma}[definition]{Lemma}
\newtheorem{theorem}[definition]{Theorem}
\newtheorem{corollary}[definition]{Corollary}

\numberwithin{equation}{section}

\DeclarePairedDelimiter{\set}{\{}{\}}
\renewcommand{\div}{\operatorname{div}}
\newcommand{\grad}{\operatorname{grad}}

\begin{document}

%
\title[Stability via closure relations]{Stability via closure relations with applications to dissipative and port-Hamiltonian systems}
\author[J. Glück]{Jochen Glück}
\address[J. Glück]{School of Mathematics and Natural Sciences, IMACM, University of Wuppertal, Gaußstraße 20, 42119 Wuppertal, Germany}
\email{glueck@uni-wuppertal.de}
\author[B. Jacob]{Birgit Jacob}
\address[B. Jacob]{School of Mathematics and Natural Sciences, IMACM, University of Wuppertal, Gaußstraße 20, 42119 Wuppertal, Germany}
\email{bjacob@uni-wuppertal.de}
\author[A. Meyer]{Annika Meyer}
\address[A. Meyer]{School of Mathematics and Natural Sciences, IMACM, University of Wuppertal, Gaußstraße 20, 42119 Wuppertal, Germany}
\email{anmeyer@uni-wuppertal.de}
\author[C. Wyss]{Christian Wyss}
\address[C. Wyss]{School of Mathematics and Natural Sciences, IMACM, University of Wuppertal, Gaußstraße 20, 42119 Wuppertal, Germany}
\email{wyss@math.uni-wuppertal.de}
\author[H. Zwart]{Hans Zwart}
\address[H. Zwart]{Department of Applied Mathematics, University of Twente, P.O.\ Box 217, 7500 Enschede, and Dynamics   and   Control   group,   Eindhoven  University  of  Technology,  5612  AZ, Eindhoven, The Netherlands.}
\email{h.j.zwart@utwente.nl}
\date{\today}
\begin{abstract}
    We consider differential operators $A$ that can be represented by means of a so-called closure relation in terms of a simpler operator $A_{\operatorname{ext}}$ defined on a larger space.
    We analyze how the spectral properties of $A$ and $A_{\operatorname{ext}}$ are related and give sufficient conditions for  exponential stability of the semigroup generated by $A$ in terms of the semigroup generated by $A_{\operatorname{ext}}$.

    As applications we study the long-term behaviour of a coupled wave-heat system on an interval, parabolic equations on bounded domains that are coupled by matrix valued potentials, and of linear infinite-dimensional port-Hamiltonian systems with dissipation on an interval.
\end{abstract}

\maketitle

\section{Introduction} 
\label{section:introduction}

In the study of linear autonomous evolution equations $\dot z(t) = A_S z(t)$ -- where $A_S$ is for instance a differential operator on a Hilbert or Banach space -- one is, beyond well-posedness, often interested in understanding the long-term behaviour of the solutions and, as a step in this direction, the spectral properties of the operator $A_S$.
In some cases the operator $A_S$ can be analyzed by using so-called \emph{closure relations}: one considers an operator $A_\ext$ that is simpler but defined on a larger space, relates this operator to $A_S$, and then uses properties of $A_\ext$ to derive desired properties of $A_S$. 
In \cite{SchwenningerZwart2014,ZwGoMa16} it was shown that this approach can, in particular, be used to show well-posedness of $\dot z(t) = A_S z(t)$ by showing that $A_S$ generates a contractive $C_0$-semigroup under appropriate assumptions.
In this article we use the same approach, but with the goal to study spectral properties of $A_S$ and the long-term behaviour of the solutions to $\dot z(t) = A_S z(t)$.

Our setting is, more concretely, as follows: we consider the block operator 
\begin{align*}
    A_\ext 
    = 
    \begin{pmatrix}
        A_{11} & A_{12} \\ 
        A_{21} & 0
    \end{pmatrix}
\end{align*}
on the product $H_1 \times H_2$ of two Hilbert spaces. 
Given $A_\ext$ and a bounded operator $S$ on $H_2$ we construct the operator $A_S \coloneqq A_{11} + A_{12} S A_{21}$ on $H_1$ and derive spectral properties of $A_S$ and asymptotic properties of the $C_0$-semigroup generated by $A_S$ from related properties of $A_\ext$. 
The construction of $A_S$ can be interpreted as follows: 
if we consider the equation 
\begin{align*}
    \begin{pmatrix}
        g_1 \\ 
        g_2
    \end{pmatrix}
    = 
    \begin{pmatrix}
        A_{11} & A_{12} \\ 
        A_{21} & 0
    \end{pmatrix}
    \begin{pmatrix}
        h_1 \\ h_2
    \end{pmatrix}
\end{align*}
and add the \emph{closure relation} $h_2 = S g_2$, we end up with the equation $g_1 = (A_{11} + A_{21} S A_{21}) h_1$ for $g_1$ and $h_1$, where the operator $A_S$ naturally occurs. 

As a simple but illuminating example let us roughly outline how one can obtain the Laplace operator on $L^2(\bbR^d)$ by this type of construction: one takes $H_1 := L^2(\bbR^d; \bbC)$, $H_2 := L^2(\bbR^d; \bbC^d)$ and
\begin{align*}
    A_\ext 
    = 
    \begin{pmatrix}
        0 &    -\operatorname{div} \\ 
        \grad &  0
    \end{pmatrix}
\end{align*}
on an appropriate domain. 
For $S = \id$ one then has $A_S = \Delta$. 
For more general $S$ one can construct divergence form elliptic operators in this way.
More involved examples are discussed in Sections~\ref{sec:wave-heat}--\ref{sec:phs}.

The detailed setting reagrding the operators $A_\ext$ and $A_S$ is discussed in Section~\ref{section:setting} and our theoretical main results about the spectrum and the long-term behavior are given in Section~\ref{sec:inheritance-of-stability}.
Afterwards we discuss three different applications of our results: 
in Section~\ref{sec:wave-heat} we analyze a wave and a heat equation on an interval that are coupled by boundary conditions. While the model itself is rather simple, it nicely demonstrates how our approach can be used to study equations that exhibit both hyperbolic and parabolic behavior on different parts of the spatial domain.
In Section~\ref{sec:coupled-parabolic-equations} we show how our results can be used to analyze a system of parabolic equations on a bounded domain which are coupled by a matrix-valued potential. 
As concrete examples, we study coupled heat equations and coupled biharmonic equations.
In the final Section~\ref{sec:phs} we apply our theory to port-Hamiltonian systems on an interval which are subject to dissipation.

For the applications to the wave-heat system (Section~\ref{sec:wave-heat}) and port-Hamiltonian systems (Section~\ref{sec:phs}) it is important that the domain block operator $A_\ext$ described above is allowed to have a somewhat subtle property:
instead of considering separate operators $A_{11}$ and $A_{12}$ in the first line of $A_\ext$ we work with a single operator $A_1$ from $H_1 \times H_2$ to $H_1$ which has a domain that cannot, in general, be split into two subspaces of $H_1$ and $H_2$ -- compare however Subsection~\ref{subsection:the-operators} for the case where such a splitting is possible.

\section{Operator theoretic setting} 
\label{section:setting}

In this section we describe the general theoretical setting for closure relations that we will use throughout the rest of the article.
We use the convention that the inner product in complex Hilbert spaces is linear in the first component and antilinear in the second.

\subsection{The operators $A_\ext$ and $A_S$}
\label{subsection:the-operators}

Throughout Sections~\ref{section:setting} and~\ref{sec:inheritance-of-stability} let $H_1, H_2$ be complex non-zero Hilbert spaces, let
\begin{align*}
    A_1: \dom(A_1) \subseteq H_1 \times H_2 \to H_1
    \quad \text{and} \quad 
    A_{21}: \dom(A_{21}) \subseteq H_1 \to H_2
\end{align*}
be linear operators, and let $S: H_2 \to H_2$ be a bounded linear operator which is coercive, meaning that there exists a number $\nu > 0$ such that
\begin{align*}
    \re \inner{h_2}{S h_2}_{H_2} \ge \nu \norm{h_2}_{H_2}^2
    \qquad \text{for all } h_2 \in H_2.
\end{align*}
So in particular, $S$ is bijective and the numerical range of $S$ (and thus also of $S^{-1}$) is contained in the right half plane of $\bbC$.

From these objects we construct an \emph{extended operator} $A_\ext$ on $H_1 \times H_2$ and an operator $A_S$ on $H_1$ that can be obtained from $A_\ext$ by adding a so-called \emph{closure relation}:
\begin{definition}[$A_\ext$ and $A_S$] 
    \label{def:operators}
    We define the following operators:
    \begin{enumerate}[label=(\alph*)]
        \item\label{def:operators:itm:A_ext}
        Let $A_\ext: \dom(A_\ext) \subseteq H_1 \times H_2 \to H_1 \times H_2$ be given by
        \begin{align*}
            \dom(A_\ext) 
            & := 
            \left\{ 
                \begin{pmatrix}
                    h_1 \\ 
                    h_2
                \end{pmatrix} 
                \in \dom(A_1) 
                \suchthat 
                h_1 \in \dom(A_{21})
            \right\},
            \\ 
            A_\ext 
            & := 
            \begin{pmatrix}
                \, A_1            \\ 
                \begin{array}{cc}
                    A_{21} & 0 \\
                \end{array}
            \end{pmatrix}    .
        \end{align*}
        \item\label{def:operators:itm:A_S} 
        Let $A_S: \dom(A_S) \subseteq H_1 \to H_1$ be the operator that has domain
        \begin{align*}
            \dom(A_S) 
            & := 
            \left\{
                h_1 \in \dom(A_{21}) 
                \suchthat 
                \begin{pmatrix}
                    h_1          \\ 
                    S A_{21} h_1
                \end{pmatrix}
                \in \dom(A_1)
            \right\}
            \\ 
            & = 
            \left\{
                h_1 \in \dom(A_{21}) 
                \suchthat 
                \begin{pmatrix}
                    h_1          \\ 
                    S A_{21} h_1
                \end{pmatrix}
                \in \dom(A_\ext)
            \right\}
        \end{align*}
        and is given by
        \begin{align*}
            A_S h_1 
            := 
            A_1 
            \begin{pmatrix}
                h_1          \\
                S A_{21} h_1
            \end{pmatrix}
        \end{align*}
        for all $h_1 \in \dom(A_S)$.
    \end{enumerate}
\end{definition}

The subsequent simple observation, which follows readily from the definitions of $A_\ext$ and $A_S$, turns out to be quite useful.
\begin{remark}
  \label{rem:A_ext-vs-A_S}
  For all $h_1 \in \dom(A_S)$ ones has
    \begin{equation}
    \label{eq:A_ext-vs-A_S}
        A_\ext 
        \begin{pmatrix}
            h_1          \\
            S A_{21} h_1
        \end{pmatrix}
        = 
        \begin{pmatrix}
            A_S h_1 \\ 
            A_{21} h_1
        \end{pmatrix}.
    \end{equation}
\end{remark}

We are mainly interested in contractive $C_0$-semigroups generated by the operators $A_\ext$ and $A_S$. 
If $A_\ext$ generates such a semigroup, then so does $A_S$; this was proved in \cite[Theorem~2.2]{ZwGoMa16}:
\begin{theorem}[Zwart, Le Gorrec, Maschke]
    \label{thm:inheritance-of-domination}
    If $A_\ext$ generates a contractive $C_0$-semigroup on $H_1 \times H_2$, then $A_S$ generates a contractive $C_0$-semigroup on $H_1$.
\end{theorem}

This generation result naturally raises the question which further properties are inherited by $A_S$ from $A_\ext$. 
In Section~\ref{sec:inheritance-of-stability} we will discuss this question in detail for spectral and stability properties.

\subsection{Splitting of $A_1$}

In some -- though not all -- concrete examples the following additional splitting assumption is satisfied:
\begin{splitting_assumption}
    \label{ass:splitting}
    The domain of $A_1$ splits as 
    \begin{align*}
        \dom(A_1) = \dom(A_{11}) \times \dom(A_{12})
    \end{align*}
    for vector subspaces $\dom(A_{11}) \subseteq H_1$ and $\dom(A_{12}) \subseteq H_2$, and hence $A_1$ can be written as 
    \begin{align*}
        A_1 = 
        \begin{pmatrix}
            A_{11} & A_{12}
        \end{pmatrix}
    \end{align*}
    for operators $A_{11}: \dom(A_{11}) \subseteq H_1 \to H_1$ and $A_{12}: \dom(A_{12}) \subseteq H_2 \to H_1$.
\end{splitting_assumption}

In cases where this assumption is satisfied, the domain and the action of $A_\ext$ can be expressed in terms of the three block operators $A_{11}$, $A_{12}$, and $A_{21}$. 
Let us write this down explicitly in the following proposition:
\begin{proposition}[$A_\ext$ and $A_S$ under the splitting assumption]
    \label{prop:splitting-repr}
    Let the splitting assumption~\ref{ass:splitting} be satisfied.
    \begin{enumerate}[label=\upshape(\alph*)]
        \item\label{prop:splitting-repr:itm:A_ext}
        The domain of $A_\ext$ is given by
        \begin{align*}
            \qquad \quad \quad & 
            \dom(A_\ext) 
            \\
            & = 
            \left\{ 
                \begin{pmatrix}
                    h_1 \\ 
                    h_2
                \end{pmatrix} 
                \in H_1 \times H_2
                \suchthat 
                h_1 \in \dom(A_{11}) \cap \dom(A_{21})
                ,\, 
                h_2 \in \dom(A_{12})
            \right\} 
            \\
            & = 
            \Big( \dom(A_{11}) \cap \dom(A_{21}) \Big) \times \dom(A_{12}),
        \end{align*}
        and $A_\ext$ acts as the block operator 
        \begin{align*}
            A_\ext 
            = 
            \begin{pmatrix}
                A_{11} & A_{12} \\ 
                A_{21} & 0
            \end{pmatrix}.
        \end{align*}

        \item\label{prop:splitting-repr:itm:A_S} 
        The the domain of $A_S$ is given by
        \begin{align*}
            \qquad \qquad
            \dom(A_S) 
            = 
            \left\{ h_1 \in \dom(A_{11}) \cap \dom(A_{21})  \suchthat  S A_{21} h_1 \in \dom(A_{12}) \right\}
        \end{align*}
        and $A_S$ acts as
        \begin{align*}
            A_S h_1 = A_{11} h_1 + A_{12} S A_{21} h_1
        \end{align*}
        for all $h_1 \in \dom(A_S)$.
    \end{enumerate}
\end{proposition}
\begin{proof}
    This follows readily from the definitions of $A_{\ext}$ and $A_S$.
\end{proof}

Note that part~\ref{prop:splitting-repr:itm:A_S} of the proposition can be rephrased by saying that, under the splitting assumption~\ref{ass:splitting}, $A_S = A_{11} + A_{12} S A_{21}$, where we use the common convention that the sum of two operators is defined on the intersection of their domains and that the composition of two operators has its maximal domain.
If, in addition to the splitting assumption~\ref{ass:splitting}, the operator $A_{11}$ is everywhere defined and bounded, one can say more about the components of $A_\ext$, see Proposition~\ref{prop:skew-hermitian-equiv}.

\subsection{Skew-adjoint properties}

Let us mention again that we are mainly interested in the case where $A_\ext$ generates a contraction semigroup and is thus, in particular, dissipative.
If the splitting assumption~\ref{ass:splitting} is satisfied, 
dissipativity of $A_\ext$ implies that $A_{21}$ and $A_{12}$ are, in a sense, skew-adjoint to each other. 
More precisely, the following holds.
\begin{theorem}
    \label{thm:skew-hermitian}
    If $A_\ext$ is dissipative and the splitting assumption~\ref{ass:splitting} is satisfied, then
    \begin{align*}
        \inner{h_1}{A_{12} h_2}_{H_1}
        = 
        - \inner{A_{21} h_1}{h_2}_{H_2}
    \end{align*}
    for all $(h_1,h_2) \in \dom(A_{\ext})$.
\end{theorem}

\begin{proof}
    Let 
    $
        \begin{psmallmatrix}
		h_1 \\ h_2
	\end{psmallmatrix} 
        \in 
        \dom(A_\ext) 
        = \bigl(\dom(A_{11}) \cap \dom(A_{21})\bigr)\times \dom(A_{12})
    $.
    The dissipativity of $A_\ext$ implies
    \begin{align*}
	\begin{split}
		0 
            & \geq 
            \re 
            \inner{
                A_\ext 
                \begin{pmatrix}
    			h_1 \\ h_2
				\end{pmatrix}
            }{
                \begin{pmatrix}
    			h_1 \\ h_2
			\end{pmatrix}
            }_{H_1 \times H_2}
            \\
		& = 
            \re 
            \bigl( 
                \inner{A_{11} h_1}{h_1}_{H_1} 
                + 
                \inner{A_{12} h_2}{h_1}_{H_1} 
                + 
                \inner{A_{21} h_1}{h_2}_{H_2}
            \bigr)
            .
	\end{split}
    \end{align*}
    Since \(h_1 \in  \dom(A_{11}) \cap \dom(A_{21})\) was arbitrary, the same inequality remains true if we substitute $h_1$ with \(t h_1\) for any $t > 0$; doing so and then dividing by $t$ yields
    \begin{equation*}
	0 \geq 
        \re 
        \bigl( 
            t \inner{A_{11} h_1}{h_1}_{H_1} 
            + 
            \inner{A_{12} h_2}{h_1}_{H_1} 
            + 
            \inner{A_{21} h_1}{h_2}_{H_2}
        \bigr)
    \end{equation*}
    for all $t > 0$. 
    By letting $t \downarrow 0$ we thus obtain
    \begin{equation*}
	0  \geq 
        \re 
        \bigl(
            \inner{A_{12} h_2}{h_1}_{H_1} 
            + 
            \inner{A_{21} h_1}{h_2}_{H_2}
        \bigr) .
    \end{equation*}
    As this holds for all \(h_1 \in \dom(A_{11}) \cap \dom(A_{21})\) it also holds for \(-h_1\), so we even obtain equality, i.e.,
    \begin{equation*}
        \re \inner{h_1}{A_{12} h_2}_{H_1}
        = 
        - \re \inner{A_{21} h_1}{h_2}_{H_2}
    \end{equation*}
    Finally, we substitute $\ui h_1$ for $h_1$ to see that one has equality of the imaginary parts, too, and thus
    \begin{equation*}
        \inner{h_1}{A_{12} h_2}_{H_1}
        = 
        - \inner{A_{21} h_1}{h_2}_{H_2},
    \end{equation*}
    as claimed.
\end{proof}

If the operator $A_{11}$ is even bounded, then Theorem~\ref{thm:skew-hermitian} implies that $A_{12}$ is indeed minus the adjoint of the operator $A_{21}$. 
To show this we use the following lemma.

\begin{lemma}
    \label{lem:splitting-nice}
    Assume that $A_\ext$ is densely defined and closed, that the splitting assumption~\ref{ass:splitting} is satisfied, and that the operator $A_{11}$ is everywhere defined and bounded.
    Then $A_{21}$ and $A_{12}$ are densely defined and closed, and we have
    \begin{equation*}
	A_\ext^* 
        = 
        \begin{pmatrix}
		A_{11}^* & A_{21}^* \\ 
            A_{12}^* & 0
	\end{pmatrix},
    \end{equation*}
    with domain $\dom(A_\ext^*) = \dom(A_{12}^*) \times \dom(A_{21}^*)$.
\end{lemma}

\begin{proof}
    \emph{Density of the domains:}
    It follows from the splitting assumption together with $\dom(A_{11}) = H_1$ that $\dom(A_\ext) = \dom(A_{21}) \times \dom(A_{12})$, see Proposition~\ref{prop:splitting-repr}\ref{prop:splitting-repr:itm:A_ext}.
    Since $\dom(A_\ext)$ is dense in $H_1 \times H_2$, we thus conclude that $\dom(A_{21})$ is dense in $H_1$ and $\dom(A_{12})$ is dense in $H_2$.

    \emph{Closedness of $A_{21}$ and $A_{12}$:}
    Since the sum of a closed and a bounded linear operator is again closed, it follows that the operator 
    \begin{align*} 
        \begin{pmatrix}
		0      & A_{12} \\ 
            A_{21} & 0
	\end{pmatrix}
    \end{align*}
    with domain $\dom(A_{21}) \times \dom(A_{12}) = \dom(A_\ext)$, is closed. 
    From this, one can easily deduce that $A_{21}$ and $A_{12}$ are closed, too.

    \emph{Representation of $A_\ext^*$:}
    If $A_{11} = 0$, it is straightforward to obtain the claimed form of $A_\ext^*$ and its domain from the definition of adjoints. 
    The claim for non-zero $A_{11}$ can be derived from this case by using the general observation that, for a densely defined operator $A$ and a bounded and everywhere defined operator $B$, one has $(A+B)^* = A^* + B^*$.
\end{proof}

\begin{proposition}
    \label{prop:skew-hermitian-equiv}
    Let the splitting assumption~\ref{ass:splitting} be satisfied and let $A_{11}$ be everywhere defined and bounded.
    The following are equivalent: 
    \begin{enumerate}[label=\upshape(\roman*)]
        \item\label{prop:skew-hermitian-equiv:itm:individual-properties} 
        The operators $A_{21}$ and $A_{12}$ are densely defined and closed and satisfy $A_{12} = - A_{21}^*$.  
        Moreover, the operator $A_{11}$ is dissipative.
        
        \item\label{prop:skew-hermitian-equiv:itm:A_ext} 
        The operator $A_\ext$ generates a contractive $C_0$-semigroup on $H_1 \times H_2$.
    \end{enumerate}
    If the equivalent assertions~\ref{prop:skew-hermitian-equiv:itm:individual-properties} and~\ref{prop:skew-hermitian-equiv:itm:A_ext} are satisfied, then 
    \begin{align*}
        \dom(A_S) 
        & = 
        \big\{
            h_1 \in \dom(A_{21}) \suchthatManual SA_{21} h_1 \in \dom(A_{21}^*)
        \big\}
        \qquad \text{and}
        \\
        A_S h_1 
        & = 
        A_{11} h_1 - A_{21}^* S A_{21} h_1 
    \end{align*}
    for all $h_1 \in \dom(A_S)$.
\end{proposition}

\begin{proof}
    \impliesProof{prop:skew-hermitian-equiv:itm:individual-properties}{prop:skew-hermitian-equiv:itm:A_ext}
    The assumptions of the corollary imply the $A_\ext$ can be written as
    \begin{align*} 
        A_\ext 
        =
        \begin{pmatrix}
		0      & A_{12} \\ 
            A_{21} & 0
	\end{pmatrix}
        +
        \begin{pmatrix}
		A_{11} & 0 \\ 
            0      & 0
	\end{pmatrix}
    ,
    \end{align*}
    where the first summand has the domain $\dom(A_{21}) \times \dom(A_{12})$ and is densely defined and skew-adjoint according to condition~\ref{prop:skew-hermitian-equiv:itm:individual-properties} and
    Lemma~\ref{lem:splitting-nice}. 
    The second summand is everywhere defined and boudned, and dissipative by~\ref{prop:skew-hermitian-equiv:itm:individual-properties}.
    Thus, $A_\ext$ generates a contractive $C_0$-semigroup by standard perturbation theory.
    
    \impliesProof{prop:skew-hermitian-equiv:itm:A_ext}{prop:skew-hermitian-equiv:itm:individual-properties}
    Dissipativity of $A_{11}$ follows by testing $A_\ext$ against vectors of the form $(h_1,0)$ from the left and the right, where $h_1 $ runs through $H_1$.
    The fact that $A_{12}$ and $A_{21}$ are densely defined and closed follows from Lemma~\ref{lem:splitting-nice}. 
    Moreover, since $A_\ext$ is dissipative, Theorem~\ref{thm:skew-hermitian} implies
    \begin{equation*}
	-A_{12} \subseteq A_{21}^* 
        .
    \end{equation*}
    So it remains to show the converse inclusion. 
    According to Lemma~\ref{lem:splitting-nice} the adjoint of $A_\ext$ is given by
    \begin{equation*}
        A_\ext^* 
        = 
        \begin{pmatrix}
		A_{11}^* & A_{21}^* \\ A_{12}^* & 0
	\end{pmatrix},
    \end{equation*}
    with domain \(\dom(A_\ext^*) = \dom(A_{12}^*) \times \dom(A_{21}^*).\)
    Since $A_\ext^*$ is dissipative and densely defined as well, we can apply Theorem~\ref{thm:skew-hermitian} to this operator, too, and obtain
    \begin{equation*}
	-A_{21}^* \subseteq A_{12}^{**} = A_{12},
    \end{equation*}
    where the last equality follows from the closedness of $A_{12}$.

    Finally, let the equivalent conditions~\ref{prop:skew-hermitian-equiv:itm:individual-properties} and~\ref{prop:skew-hermitian-equiv:itm:A_ext} be satisfied. 
    The claimed formulas for $\dom(A_S)$ and $A_S$ then follow from Proposition~\ref{prop:splitting-repr}\ref{prop:splitting-repr:itm:A_S} and from $A_{12} = -A_{21}^*$.
\end{proof}

\subsection{Relation to form methods in the parabolic case}
\label{subsec:forms}

In this subsection we will, in the situation of Proposition~\ref{prop:skew-hermitian-equiv}, relate the operator $A_S$ to the study of parabolic equations via form methods. 
By doing so we will show that the semigroup $(e^{tA_S})_{t \ge 0}$ generated by $A_S$ is analytic under the assumptions of Proposition~\ref{prop:skew-hermitian-equiv}. 
Therefore, it is reasonable to think of the Cauchy problem $\dot u = A_S u$ as a parabolic equation whenever the assumptions of Proposition~\ref{prop:skew-hermitian-equiv} are satisfied, and we will from now on refer to this situation as \emph{the parabolic case}.

Many linear parabolic equations on Hilbert spaces can be studied by means of \emph{sesqui-linear forms} that are related to the generator of the analytic semigroup that governs the equation. 
In fact, those analytic $C_0$-semigroups whose generators stem from a sequilinear form can be precisely characterized, see e.g.\ \cite[Section~5.3.4]{Ar04}. 
If one allows for equivalent renorming of the Hilbert space, the class of operators that can be treated by form methods becomes even a bit larger, see e.g.\ \cite[Section~5.3.5]{Ar04}.
In the following we provide more context to our abstract framework by describing its relation to sesquilinear forms.

There are two different -- though essentially equivalent -- approaches to constructing operators from sesquilinear forms: 
an algebraic approach where the form is a priori considered an algebraic object and is endowed with a norm only later on (see e.g.\ \cite[Chapter~VI]{Kato1980}) and an approach that assumes the form domain to be a Hilbert space from the beginning (see e.g.\ \cite[Section~XVIII.3]{DautrayLions2000Vol5}).
Let us briefly outline how the latter approach works; we follow \cite[Section~5.3]{Ar04}.
Let $H$ be a complex Hilbert space. 
A \emph{closed form} on $H$ is a pair $(a,V)$ with the following properties: 

\begin{enumerate}[label=\upshape(\alph*)]
    \item  
    The space $V$ carries a norm $\norm{\argument}_V$ that renders $V$ a Hilbert space. 
    Moreover, $V$ is a dense vector subspace of $H$ and the inclusion map
    \begin{align*}
        (V, \norm{\argument}_V) \hookrightarrow (H, \norm{\argument}_H)
    \end{align*}
    is continuous. 

    \item 
     The object $a$ is a sesquilinear map $a: V \times V \to \bbC$ -- meaning that it is linear in the first component and antilinear in the second -- that is continuous with respect to the norm $\norm{\argument}_V$, i.e., there exists a number $M \ge 0$ such that the estimate $\modulus{a(u,v)} \le M \norm{u}_V \norm{v}_V$ holds for all $u,v \in V$.

     \item 
     The form $(a, V)$ is \emph{elliptic}, meaning that there exist numbers $\alpha > 0$ and $\omega \in \bbR$ such that 
    \begin{align*}
        \re a(u,u) + \omega \norm{u}_H^2 
        \ge 
        \alpha \norm{u}_V^2
    \end{align*}
    for every $u \in V$. 
\end{enumerate}

Now consider the closed form $(a, V)$ on $H$.
We can define the closed linear operator $A: \dom(A) \subseteq H \to H$ by
\begin{align*}
\dom(A) &= \{u \in V \suchthatManual \text{ there is } f \in H \text{ such that } a(u,v) = \inner{f}{v}_H \text{ for all } v \in V\}, \\ 
    Au &= f,
\end{align*}
where the vector $f$ in the second line is the same vector that occurs within the set in the first line 
(note that $f$ is uniquely determined as $V$ is, by assumption, dense in $H$).
The operator $-A$ generates an analytic $C_0$-semigroup on $H$ \cite[Section~5.3.1]{Ar04}.

In the framework described in Section~\ref{section:setting} one can describe the operator $A_S$ by means of form methods, provided that the splitting assumption~\ref{ass:splitting} is satisfied and that the operator block $A_{21}$ satisfies a certain lower estimate.
In the following theorem we endow $\dom(A_{21})$ with the graph norm of $A_{21}$ given by $\norm{u}_{A_{21}}^2 \coloneqq \norm{A_{21}u}_{H_2}^2 + \norm{u}_{H_1}^2$ for all $u \in \dom(A_{21})$.
\begin{theorem}
    \label{thm:via-forms}
    Let the splitting assumption~\ref{ass:splitting} be satisfied and let $A_{11}$ be everywhere defined and bounded. 
    Assume moreover that $A_\ext$ generates a contractive $C_0$-semigroup on $H_1 \times H_2$. 
    Define the sesquilinear map $a: \dom(A_{21}) \times \dom(A_{21}) \to \bbC$ by 
    \begin{align*}
        a(u,v) 
        \coloneqq 
        \inner{S A_{21}u}{A_{21}v}_{H_2} 
        - 
        \inner{A_{11}u}{v}_{H_1}
    \end{align*}
    for all $u,v \in \dom(A_{21})$.
    
    Then $(a, \dom(A_{21}))$ is a closed form on $H_1$ and
    the operator on $H_1$ associated to $(a, \dom(A_{21}))$ is $-A_S$. 
    Hence, the $C_0$-semigroup $(e^{tA_S})_{t \ge 0}$ on $H_1$ is analytic.
\end{theorem}

\begin{proof}
    We first show that $\big(a, \dom(A_{21})\big)$ is a closed form on $H_1$:
    
    The graph norm $\norm{\argument}_{A_{21}}$ on $\dom(A_{21})$ clearly stems from an inner product; 
    it is also complete since $A_{21}$ is closed according to Proposition~\ref{prop:skew-hermitian-equiv}.
    The same proposition shows that $A_{21}$ is densely defined on $H_1$, so the form domain $\dom(A_{21})$ of $a$ is indeed dense in $H_1$.
    Continuity of the embedding of $\big(\dom(A_{21}), \norm{\argument}_{A_{21}}\big)$ into $(H_1, \norm{\argument}_{H_1})$ is obvious from the definition of the graph norm $\norm{\argument}_{A_{21}}$. 
    Continuity of $a$ with respect to the norm $\norm{\argument}_{A_{21}}$ is immediate.
    Ellipticity of the form follows from the coercivity assumption on $S$ (see the beginning of Subsection~\ref{subsection:the-operators}) and the continuity of $A_{11}$.
    
    Now let $A: \dom(A) \subseteq H_1 \to H_1$ denote the operator that is associated to the closed and elliptic form $\big(a, \dom(A_{21})\big)$.  
    By using the form of the operator $A_S$ given in Proposition~\ref{prop:skew-hermitian-equiv} along with the very definition of $A$, one can check via a straightforward computation that, indeed, $A = -A_S$.
    So in particular, the semigroup generated by $A_S$ is analytic.
\end{proof}

Another condition for $A_S$ to generate an analytic semigroup is (even in the setting of Banach spaces) given in \cite[Theorem~2.2]{SchwenningerZwart2014}. 
That result even gives that the angle of analyticity is $\frac{\pi}{2}$. 
To check the assumption of that theorem one needs multiplicative perturbation theory of $C_0$-groups. 
In contrast, Theorem~\ref{thm:via-forms} yields analyticity in the general situation of Proposition~\ref{prop:skew-hermitian-equiv}, without any additional assumption. 
In this situation it can happen that the angle of analyticity of $(e^{tA_S})_{t \ge 0}$ is strictly less than $\frac{\pi}{2}$, as the following simple example shows:

\begin{example}
    \label{exa:analytic-smaller-angle}
    Let $H_1 = H_2 = \ell^2$ and consider the skew-adjoint operator $A_{21} = A_{12}$ on $\ell^2$ that is given by multiplication with the unbounded sequence $(\ui n)_{n \in \bbN}$ and is endowed with its maximal domain $\dom(A_{21})$. 
    Moreover, we let $A_{11} = 0$. 
    Finally, set $S \coloneqq (1+\ui) \id_{\ell^2}$. 

    Then we are in the setting described in Subsection~\ref{subsection:the-operators} and the splitting assumption~\ref{ass:splitting} is satisfied. 
    According to Proposition~\ref{prop:skew-hermitian-equiv} the operator $A_\ext$ generates a contractives $C_0$-semigroup on $\ell^2 \times \ell^2$. 
    All assumptions of Theorem~\ref{thm:via-forms} are thus satisfied.

    The operator $A_S$ is given by multiplication with the sequence $\big(-(1+\ui)n^2\big)_{n \in \bbN}$ (and has the maximal domain of this multiplication operator). 
    Hence, the angle of analyticity of $(e^{tA_S})_{t \ge 0}$ is strictly less than $\frac{\pi}{2}$.
\end{example}

Note, however, that if the operator $S$ is self-adjoint, then it follows that $A_S$ is also self-adjoint and hence generates an analytic $C_0$-semigroup of angle $\frac{\pi}{2}$.

\section{Spectrum and stability of $A_S$}
\label{sec:inheritance-of-stability}

In this section we prove a variety of results about the inheritance of spectral and stability properties from $A_\ext$ to $A_S$.

\subsection{Inheritance of spectral properties}

The purpose of this subsection is to analyse how purely imaginary spectral values of $A_\ext$ and $A_S$ are related.
All our results to this end are based on the following lemma.
\begin{lemma}
    \label{lem:approx-sequence}
    Assume that $A_\ext$ is dissipative. 
    Let $(\ui\omega_n) \subseteq \ui\bbR$ and $(v_n) \subseteq \dom(A_S)$ be sequences such that $\norm{v_n}_{H_1} = 1$ for all $n$ and
    \begin{align*}
        (\ui\omega_n - A_S) v_n \to 0
        \qquad \text{in } H_1.
    \end{align*}
    Then $A_{21} v_n \to 0$ in $H_2$.
\end{lemma}

\begin{proof}
    First note that $\dom(A_S) \subseteq \dom(A_{21})$ by the definition of $A_S$, so one can indeed apply $A_{21}$ to $v_n$.
    Define the vectors 
    \begin{align*}
        z_n := 
        \begin{pmatrix}
            v_n          \\ 
            S A_{21} v_n
        \end{pmatrix}
        \in \dom(A_\ext) \subseteq \dom(A_1).
    \end{align*}
    Their norms satisfy $\norm{z_n}_{H_1 \times H_2}^2 = 1 + \norm{S A_{21} v_n}_{H_2}^2$. 
    Moreover, we have 
    \begin{align*}
        \left(
            \begin{pmatrix}
                \ui \omega_n & 0      \\ 
                0            & S^{-1} 
            \end{pmatrix}
            -
            A_\ext
        \right)
        z_n
        = 
        \begin{pmatrix}
            (\ui\omega_n - A_S) v_n \\ 
            0
        \end{pmatrix}
        \to 
        0 
        \quad \text{in } H_1 \times H_2.
    \end{align*}
    By multiplying this with $1/\norm{z_n}_{H_1 \times H_2} \le 1$, testing against the bounded sequence $\big( z_n / \norm{z_n}_{H_1 \times H_2} \big)$ from the left, and taking real parts, we obtain
    \begin{align*}
        \frac{1}{\norm{z_n}_{H_1 \times H_2}^2} 
        \Big(
            \underbrace{
                \re \inner{S A_{21} v_n}{ A_{21} v_n}_{H_2}
            }_
            {
                \ge \nu \norm{A_{21} v_n}_{H_2}^2
            }
            + 
            \underbrace{
                \re \inner{z_n}{-A_\ext z_n}_{H_1 \times H_2}
            }_ 
            {\ge 0}
        \Big)
        \to   0,
    \end{align*}
    where the second real part is $\ge 0$ since $A_\ext$ is dissipative according to the assumptions of the lemma. 
    Thus, 
    \begin{align*}
        0 
        \leftarrow 
        \frac{\norm{A_{21} v_n}_{H_2}^2}{\norm{z_n}_{H_1 \times H_2}^2} 
        \ge 
        \frac{
            \norm{A_{21} v_n}_{H_2}^2
        }{  
            1 + \norm{S}^2 \norm{A_{21} v_n}_{H_2}^2
        }. 
    \end{align*} 
    This implies that $\norm{A_{21} v_n}_{H_2}$ converges to $0$, since the continuous function 
    \begin{align*}
        [0, \infty) \to     [0,\infty),
        \qquad 
        x           \mapsto \frac{x}{1 + \norm{S}^2 x}
    \end{align*}
    vanishes only at $x = 0$ and converges to $1/\norm{S}^2 \not= 0$ as $x \to \infty$.
\end{proof}

Now we analyze the eigenvalues of $A_S$ and $A_\ext$ on the imaginary axis; 
general spectral values are treated in Theorem~\ref{thm:spectrum}.
\begin{theorem}[Peripheral point spectrum of $A_\ext$ and $A_S$]
    \label{thm:eigenvalues}
    Assume that $A_\ext$ is dissipative. 
    Then the following assertions are equivalent for $\omega \in {\mathbb R}$ and $v \in H_1$:
    \begin{enumerate}[label=\upshape(\roman*)]
        \item\label{thm:eigenvalues:itm:A_S} 
        $v \in \dom(A_S)$ and $A_S v = \ui\omega v$; 
        
        \item\label{thm:eigenvalues:itm:A_ext}  
        $
            \begin{pmatrix}
                v \\ 
                0
            \end{pmatrix}
            \in \dom(A_\ext)
        $
        and
        $
            A_\ext 
            \begin{pmatrix}
                v \\ 
                0
            \end{pmatrix}
            = 
            \ui\omega 
            \begin{pmatrix}
                v \\ 
                0
            \end{pmatrix}
        $;
        \item\label{thm:eigenvalues:itm:componentwise} 
        $v \in \dom(A_{21})$ and $A_{21} v = 0$, 
        as well as 
        $
            \begin{pmatrix}
                v \\ 
                0
            \end{pmatrix}
            \in \dom(A_1)
        $
        and 
        $
            A_1
            \begin{pmatrix}
                v \\ 
                0
            \end{pmatrix}
            = 
            \ui\omega v 
        $.
    \end{enumerate}
    In particular, the point spectra of $A_S$ and $A_\ext$ satisfy
    \begin{equation*}
        \specPnt(A_S) \cap \ui\bbR 
        \subseteq 
        \specPnt(A_\ext) \cap \ui\bbR.
    \end{equation*}
\end{theorem}

\begin{proof}
    \impliesProof{thm:eigenvalues:itm:A_S}{thm:eigenvalues:itm:A_ext}
    This implication is trivial if $v = 0$, so we may, and shall, assume that $\norm{v}_{H_1} = 1$.
    If we apply Lemma~\ref{lem:approx-sequence} to the constant sequences $(\ui\omega)$ and $(v)$ we obtain $A_{21} v = 0$, 
    and thus
    $
        \begin{pmatrix}
            v \\ 
            0
        \end{pmatrix}
        = 
        \begin{pmatrix}
            v          \\ 
            S A_{21} v
        \end{pmatrix}
        \in 
        \dom(A_{\ext})
    $
    and 
    \begin{align*}
        A_\ext 
        \begin{pmatrix}
            v \\ 
            0
        \end{pmatrix}
        = 
        A_\ext 
        \begin{pmatrix}
            v          \\ 
            S A_{21} v
        \end{pmatrix}
        = 
        \begin{pmatrix}
            A_S v    \\ 
            A_{21} v
        \end{pmatrix}
        = 
        \ui\omega
        \begin{pmatrix}
            v \\ 
            0
        \end{pmatrix} ,
    \end{align*}
    where the second equality follows from Remark~\ref{rem:A_ext-vs-A_S}. 
    
    \impliesProof{thm:eigenvalues:itm:A_ext}{thm:eigenvalues:itm:componentwise} 
    and 
    \impliesProof{thm:eigenvalues:itm:componentwise}{thm:eigenvalues:itm:A_S}
    These implications easily follow from the definitions of $A_\ext$ and $A_S$ and from Remark~\ref{rem:A_ext-vs-A_S}.
\end{proof}

\begin{corollary}
    \label{cor:non-eigenvalues}
    If $A_\ext$ is dissipative and $A_{21}$ is injective, then $A_S$ does not have any eigenvalues in $\ui\bbR$.
\end{corollary}

\begin{corollary}
    \label{cor:splitting-eigenvalues}
    If $A_\ext$ is dissipative and the splitting assumption~\ref{ass:splitting} is satisfied, then
    \begin{align*}
        \ker(\ui \omega - A_S) = \ker A_{21} \cap \ker (\ui \omega - A_{11})
    \end{align*}
    for every $\omega \in \bbR$.
\end{corollary}

\begin{proof}
    This is a consequence of the characterization of eigenvalues of $A_S$ in Theorem~\ref{thm:eigenvalues}\ref{thm:eigenvalues:itm:A_S} and~\ref{thm:eigenvalues:itm:componentwise} and of the splitting assumption~\ref{ass:splitting}.
\end{proof}

Under the assumptions of Theorem~\ref{thm:via-forms} -- which are bit stronger than those of Corollary~\ref{cor:splitting-eigenvalues} -- one can give an alternative proof of Corollary~\ref{cor:splitting-eigenvalues} that relies on the form $\big(a, \dom(A_{21})\big)$ from Theorem~\ref{thm:via-forms}. 
It seems worthwhile to write down this argument explicitly.

\begin{proof}[Alternative proof of Corollary~\ref{cor:splitting-eigenvalues} under the assumptions of Theorem~\ref{thm:via-forms}]
    We use the formula for $A_S$ from Proposition~\ref{prop:skew-hermitian-equiv} and the form $a: \dom(A_{21}) \times \dom(A_{21}) \to \bbC$ on $H_1$ from Theorem~\ref{thm:via-forms}.
    
    ``$\supseteq$''
    Let $u \in \ker A_{21} \cap \ker(\ui \omega - A_{11})$.
    Then $a(u,v) = - \inner{\ui \omega u}{v}$ for each $v \in \dom(A_{21})$, 
    so $u \in \dom(A_S)$ and $A_S u = \ui \omega u$.

    ``$\subseteq$''
    Let $u \in \ker(\ui\omega - A_S)$. 
    Then $u \in \dom(A_S)$, so $u \in \dom(A_{21})$ and $S A_{21} u \in \dom(A_{21}^*)$.
    By testing the equality $-A_S u = -\ui \omega u$ against $u$ we obtain
    \begin{align*}
        -\ui \omega \norm{u}_{H_1}^2 
        = 
        -\inner{A_S u}{u}_{H_1}
        =
        a(u,u) 
        = 
        \inner{S A_{21} u}{A_{21} u} - \inner{A_{11}u}{u}_{H_1}
        .
    \end{align*}
    By taking the real part and using the dissipativity of $A_{11}$ (which follows from Proposition~\ref{prop:skew-hermitian-equiv}) and the coercivity assumption on $S$ 
    (see the beginning of Subsection~\ref{subsection:the-operators}) we then get
    $
        0 
        \ge 
        \nu \norm{A_{21} u}_{H_2}^2
    $
    and hence, $u \in \ker A_{21}$.
    Thus, 
    \begin{align*}
        \ui \omega u 
        = 
        A_S u 
        = 
        A_{11} u - A_{21^*} S A_{12} u 
        = 
        A_{11} u,
    \end{align*}
    so also $u \in \ker(\ui \omega - A_{11})$.
\end{proof}

This proof is essentially an abstract version of the argument used in \cite[Prop.~3.6]{DoGl22} to study the imaginary point spectrum of systems of heat equations coupled by a matrix-valued potential. 
The proof allows an approach to studying the imaginary eigenvalues of $A_S$ that does not rely on $A_\ext$: 
instead of defining and working with $A_\ext$ one can start with the operators $A_{21}$, $A_{11}$, and $S$, define the form $\big(a, \dom(A_{21})\big)$ from those operators, and show -- under appropriate assumptions on the involved operators -- that the form is closed and elliptic and that the associated operator is $-A_S$, where $A_S$ is given as in Proposition~\ref{prop:skew-hermitian-equiv}.
Then the above proof can be applied without reference to $A_\ext$.

Let us explain in more detail what, intuitively speaking, the point of Corollary~\ref{cor:splitting-eigenvalues} is. 
Due to the splitting assumption that we require in Corollary~\ref{cor:splitting-eigenvalues}, the operator $A_S$ acts, according to Proposition~\ref{prop:splitting-repr}\ref{prop:splitting-repr:itm:A_S}, as $A_S h_1 = A_{11} h_1 + A_{12} S A_{21} h_1$ on $\dom(A_S)$. 
So Corollary~\ref{cor:splitting-eigenvalues} tells us, in particular, that adding the operator $A_{12}SA_{21}$ to $A_{11}$ cannot produce imaginary eigenvalues that are not already present for $A_{11}$.
This is due to the special structure of the operator $A_{12} S A_{21}$ and due to the assumptions that we impose on these operators (by assuming that $A_\ext$ is dissipative). 
In general, adding two dissipative operators can generate imaginary eigenvalues that are not present for any of the single operators, as the following simple three dimensional example illustrates.

\begin{example}
    \label{exa:turing-instab}
    Consider the matrices 
    \begin{align*}
        A 
        \coloneqq
        \begin{pmatrix}
            0  & -1 & 1 \\ 
            1  & 0  & 0 \\ 
            -1 & 0  & -1
        \end{pmatrix}
        \qquad \text{and} \qquad 
        B 
        \coloneqq 
        \begin{pmatrix}
            0  & 2 & -1 \\ 
            -2 & 0 &  0 \\ 
            1  & 0 & -1
        \end{pmatrix}
    \end{align*}
    in $\bbR^{3 \times 3}$. 
    Both matrices are dissipative and a brief numerical computation shows that every eigenvalue of both matrices has strictly negative real part. 
    However, the sum
    \begin{align*}
        A + B 
        = 
        \begin{pmatrix}
             0 & 1 & 0 \\ 
            -1 & 0 & 0 \\ 
             0 & 0 & -2
        \end{pmatrix}
    \end{align*}
    has the spectrum $\{-2,\ui, -\ui\}$, i.e., two of its eigenvalues are located on the imaginary axis.
\end{example}

The existence of purely imaginary eigenvalues for the sum of two operators, despite both operators having eigenvalues with strictly negative real part only, is sometimes referred to as \emph{Turing instability} in reference to the more general phenomenon that adding two differential equations can destroy stability properties of the individual systems.  
So Corollary~\ref{cor:splitting-eigenvalues} shows that Turing instability cannot occur when adding $A_{11}$ and $A_{12} S A_{21}$.

A similar result as in Theorem~\ref{thm:eigenvalues} can  be proved for spectral values rather than eigenvalues; 
this is our next theorem.
\begin{theorem}[Peripheral spectrum of $A_\ext$ and $A_S$]
    \label{thm:spectrum}
    If $A_\ext$ generates a contraction semigroup on $H_1 \times H_2$, then
    \begin{align*}
        \spec(A_S) \cap \ui\bbR 
        \subseteq 
        \spec(A_\ext) \cap \ui\bbR
        .
    \end{align*}
\end{theorem}

\begin{proof}
    Let $\ui \omega \in \ui\bbR$ be a spectral value of $A_S$.
    Since $A_\ext$ generates a contraction semigroup, so does $A_S$ according to Theorem~\ref{thm:inheritance-of-domination}. 
    Hence, $\ui\omega$ is in the topological boundary of $\spec(A_S)$ and is thus an approximate eigenvalue of $A_S$ \cite[Proposition~IV.1.10 on p.\,242]{EngelNagel2000}. 
    So there exists a sequence $(v_n) \subseteq \dom(A_S)$ such that $\norm{v_n}_{H_1} = 1$ for each $n$ and such that
    \begin{align*}
        (\ui\omega - A_S) v_n \to 0.
    \end{align*}
    We can now apply Lemma~\ref{lem:approx-sequence} to the constant sequence $(\ui \omega)$ and to the sequence $(v_n)$ and thus obtain $A_{21} v_n \to 0$. 
    Consequently
    \begin{align*}
        (\ui\omega - A_\ext)
        \begin{pmatrix}
            v_n          \\ 
            S A_{21} v_n
        \end{pmatrix}
        = 
        \begin{pmatrix}
            (\ui\omega - A_S) v_n                \\ 
            \ui \omega S A_{21} v_n - A_{21} v_n 
        \end{pmatrix}
        \to 0,
    \end{align*}
    where the equality follows from Remark~\ref{rem:A_ext-vs-A_S}
    and the convergence of the second component is due to $A_{21}v_n \to 0$ together with the boundedness of the operator $S$. 
    Since 
    \begin{align*}
        \norm{
            \begin{pmatrix}
                v_n          \\ 
                S A_{21} v_n
            \end{pmatrix}
        }_
        {H_1 \times H_2}
        \ge 1
    \end{align*}
    for all $n$, this implies that $\ui\omega$ is also an approximate eigenvalue, and thus a spectral value, of $A_\ext$.
\end{proof}

The following proposition uses ideas of the proof of \cite[Theorem~2.2]{ZwGoMa16} to show that having a compact resolvent is transmitted from $A_\ext$ to $A_S$.
 \begin{proposition}\label{prop: compact-resolvent}
    Let $A_\ext$ be the generator of a contraction semigroup and have a compact resolvent, then $A_S$ has compact resolvent as well.
\end{proposition}
\begin{proof}
    First, we define the following operator matrix on $H_1 \times H_2$ for $\lambda \in \bbC$
    \begin{align*}
        P_\lambda = 
        \begin{pmatrix}
            0 & 0 \\
            0 & -S^{-1} + \lambda \id
        \end{pmatrix},
    \end{align*}
    where $S: H_2 \to H_2$ is the coercive linear operator from Section \ref{subsection:the-operators}.
    For $\lambda \in \rho(A_\ext)$, we have
    \begin{align*}
        (\lambda -(A_\ext +P_\lambda))^{-1} = (\lambda -A_\ext)^{-1} + (\lambda -A_\ext)^{-1} P_\lambda (\lambda -(A_\ext +P_\lambda))^{-1}.
    \end{align*}
    Thus, by assumption we get that $A_\ext + P_\lambda$ has compact resolvent.
    For $y \in H_1$ and $(h_1,h_2) \in H_1 \times H_2$ we have 
    \begin{align*}
        (\lambda -(A_\ext+P_\lambda))
        \begin{pmatrix}
            h_1 \\ 
            h_2
        \end{pmatrix}
        = 
        \begin{pmatrix}
            y \\ 
            0
        \end{pmatrix}
    \end{align*}
    if and only if
    \begin{align*}
        y &= (\lambda -A_S) h_1 \text{ and } \\
        h_2 &= SA_2h_1.
    \end{align*}
    Thus, for every $\lambda \in \rho(A_\ext)\cap \rho(A_S)$ and every $y \in H_1$ one has
    \begin{align*}
        (\lambda -A_S)^{-1} y = \pi_1 (\lambda -(A_\ext +P_\lambda))^{-1}
        \begin{pmatrix}
            y \\
            0
        \end{pmatrix}
    \end{align*}
    From this equality we see that $A_S$ has compact resolvent.
\end{proof}

\subsection{Exponential stability}

We call a $C_0$-semigroup \emph{exponentially stable} if it converges to $0$ with respect to the operator norm as time tends to $\infty$.
In addition to exponential stability we will also use the following notion, which is -- for semigroup generators -- stronger:
\begin{definition}
    \label{def:strict-dissipativity}
    A linear operator $A: \dom(A) \subseteq H \to H$ on a complex Hilbert space $H$ is called \textit{strictly dissipative} if there exists $\varepsilon > 0$ such that
    \begin{equation*}
	\re \inner{Ah}{h} \leq -\varepsilon \norm{h}^2
    \end{equation*}
    for all $h \in \dom(A)$.
\end{definition}
If $A: \dom(A) \subseteq H \to H$ is the generator of a $C_0$-semigroup $(e^{tA})_{t \ge 0}$ on the complex Hilbert space $H$, then strict dissipativity of $A$ implies that $\norm{e^{tA}} \le e^{-t\varepsilon}$ for all $t \ge 0$ where $\varepsilon > 0$ is the number from Definition~\ref{def:strict-dissipativity}.
Hence, strict dissipativity of semigroup generators does indeed imply exponential stability of the semigroup (but the converse is not true, as one can easily observe even on two-dimensional spaces).

We now give a variety of sufficient criteria for the exponential stability of $(e^{tA_S})_{t \ge 0}$ in Theorems~\ref{thm:h_1-estimates} and~\ref{thm:empty-peripheral-spec}. 
The first of the theorems is quite elementary to prove.
The assumption that $A_{21}$ is bounded below in part~\ref{thm:h_1-estimates:itm:poincare} can be interpreted as an abstract version of the Poincaré inequality.
\begin{theorem}
    \label{thm:h_1-estimates}
    Assume that the operator $A_\ext$ generates a contraction semigroup on $H_1 \times H_2$. 
    Any of the following two assumptions implies that $A_S$ is strictly dissipative (and thus that the semigroup $(e^{tA_S})_{t \ge 0}$ is exponentially stable).
    \begin{enumerate}[label=\upshape(\arabic*)]
        \item\label{thm:h_1-estimates:itm:dissipative} 
        There exists $\varepsilon > 0$ such that 
        \begin{equation*}
		  \re \inner{A_\ext h}{h} \leq - \varepsilon \norm{h_1}_{H_1}^2
        \end{equation*}
        for all $h = (h_1,h_2) \in \dom(A_\ext)$.
        
        \item\label{thm:h_1-estimates:itm:poincare} 
	The operator $A_{21}$ is \emph{bounded below}, i.e., there exists a constant $c > 0$ such that 
        \begin{align*}
    	\norm{A_{21} h_1}_{H_2} \geq c \norm{h_1}_{H_1}
        \end{align*}
        for all $h_1 \in \dom(A_{21})$.
    \end{enumerate}
\end{theorem}

\begin{proof}
    It follows from \eqref{eq:A_ext-vs-A_S} that, for all $h_1 \in \dom(A_S)$,
    \begin{align}  
        \label{eq:thm:h_1-estimates:general-formula}
        \begin{split}
            \inner{A_S h_1}{h_1}_{H_1} 
            & = 
            \inner{
                A_\ext 
                \begin{pmatrix}
                    h_1 \\ S A_{21} h_1 
                \end{pmatrix}
            }{
                \begin{pmatrix}
                    h_1 \\ 0
    		\end{pmatrix}
            }_{H_1 \times H_2}
            \\
            & = 
            \inner{
                A_\ext 
                \begin{pmatrix}
                    h_1 \\ S A_{21} h_1 
		      \end{pmatrix}
            }{
                \begin{pmatrix}
                    h_1 \\ S A_{21} h_1
		      \end{pmatrix}
            }_{H_1 \times H_2}
            - 
            \inner{
                A_{21}h_1
            }{
                S A_{21} h_1
            }_{H_2}
        \end{split}
    \end{align}
    Now we consider both cases~\ref{thm:h_1-estimates:itm:dissipative} and~\ref{thm:h_1-estimates:itm:poincare} separately:

    \ref{thm:h_1-estimates:itm:dissipative}
    As $S$ is coercive, all values in its numerical range have real part $\ge 0$. 
    So by taking real parts in~\eqref{eq:thm:h_1-estimates:general-formula} and using condition \ref{thm:h_1-estimates:itm:dissipative} we obtain
    \begin{align*}
        \re \inner{A_S h_1}{h_1}_{H_1} 
        \le 
        -\varepsilon \norm{h_1}_{H_1}^2
    \end{align*}
    for all $h_1 \in \dom(A_S)$, which proves the strict dissipativity of $A_S$.
    
    \ref{thm:h_1-estimates:itm:poincare}
    As $A_\ext$ is dissipative and $S$ is coercive, taking real parts in~\eqref{eq:thm:h_1-estimates:general-formula} and using condition~\ref{thm:h_1-estimates:itm:poincare} yields 
    \begin{align*}
        \re \inner{A_S h_1}{h_1}_{H_1} 
        \le
        -\nu \norm{A_{21} h_1}^2_{H_2} 
        \le 
        -\nu c^2 \norm{h_1}_{H_1}^2
    \end{align*}
    for all $h_1 \in \dom(A_S)$, which shows that $A_S$ is strictly dissipative.
\end{proof}

For our second theorem on exponential stability of $A_\ext$ (Theorem~\ref{thm:empty-peripheral-spec} below) we use two auxiliary results: 
the first one, Lemma~\ref{lem:gp-lemma},
is a characterization of exponential stability. 
It follows easily from the Gearhart--Prüss theorem \cite[Theorem~V.1.11 on p.\,302]{EngelNagel2000}, so we omit its proof.
The second one, Lemma~\ref{lem:approx-sequence-bounded}, relates Lemma~\ref{lem:approx-sequence} to imaginary spectral values of $A_\ext$.
\begin{lemma}
    \label{lem:gp-lemma}
    Let $A$ be the generator of a $C_0$-semigroup on a Hilbert space $H$ with growth bound at most $0$. 
    Then the following are equivalent: 
    \begin{enumerate}[label=\upshape(\roman*)]
    \item\label{lem:gp-lemma:itm:not-stable}
     The semigroup $(e^{tA})_{t \ge 0}$ is not exponentially stable.
    \item\label{lem:gp-lemma:itm:sequences} 
        There exist sequences \((\ui\omega_n) \subseteq \ui\bbR\) and \((v_n) \subseteq \dom(A)\) such that \(\norm{v_n}=1\) for all $n$ and such that
        \begin{equation*}
	   \lim_{n \to \infty} (i\omega_n - A)v_n = 0.
        \end{equation*}
    \end{enumerate}
\end{lemma}

\begin{lemma}
    \label{lem:approx-sequence-bounded}
    Assume that $A_\ext$ generates a contraction semigroup on $H_1 \times H_2$. 
    Let $(\ui\omega_n) \subseteq \ui\bbR$ be a sequence of numbers and let $(v_n) \subseteq \dom(A_S)$ be a sequence of vectors such that $\norm{v_n}_{H_1} = 1$ for all $n$ and such that 
    \begin{align*}
        (\ui\omega_n - A_S) v_n \to 0
        \qquad \text{in } H_1.
    \end{align*}
    If the sequence $(\ui\omega_n)$ is bounded, then $\spec(A_\ext) \cap \ui \bbR \not= \emptyset$.
\end{lemma}

Note that there are two differences between the assumptions in this lemma and in Lemma~\ref{lem:approx-sequence}: 
we now assume that $A_\ext$ generates a contraction semigroup (rather than being merely dissipative) and that the sequence $(\ui\omega_n)$ is bounded.

\begin{proof}[Proof of Lemma~\ref{lem:approx-sequence-bounded}]
    First note that $A_{21}v_n \to 0$ in $H_2$ due to Lemma~\ref{lem:approx-sequence}.
    After replacing $(\ui\omega_n)$ and $(v_n)$ with subsequences, we may assume that $\ui\omega_n \to \ui \omega$ for a point $\ui\omega \in \ui\bbR$. 
    From
    \begin{align*}
        (\ui \omega - A_S) v_n 
        = 
        (\ui \omega - \ui\omega_n) v_n + (\ui \omega_n - A_S) v_n
        \to 0
    \end{align*}
    we see that
    $\ui \omega$ is an approximate eigenvalue (and thus a spectral value) of $A_S$.
    The claim thus follows from Theorem~\ref{thm:spectrum}.
\end{proof}

Now we can prove our second theorem on exponential stability.
By $\numRange(A) \coloneqq \left\{\inner{v}{Av} \suchthat v \in \dom(A), \norm{v} =1\right\}$ we denote the \emph{numerical range} of an operator $A$.
\begin{theorem}
    \label{thm:empty-peripheral-spec}
    Assume that $A_\ext$ generates a contraction semigroup and that
    \begin{align*}
        \spec(A_\ext) \cap \ui \bbR = \emptyset.
    \end{align*}
    Then any of the following assumptions implies that $(e^{tA_S})_{t \ge 0}$ is exponentially stable:
    \begin{enumerate}[label=\upshape(\arabic*)]
        \item\label{thm:empty-peripheral-spec:itm:compact-embedding} 
        The space $\dom(A_{21})$, endowed with the graph norm of $A_{21}$, embeds compactly into $H_1$.
        
        \item\label{thm:empty-peripheral-spec:itm:ev-norm-cont} 
        The semigroup $(e^{tA_S})_{t \ge 0}$ is eventually norm continuous.
        
        \item\label{thm:empty-peripheral-spec:itm:numerical-range} 
        The splitting assumption~\ref{ass:splitting} on the domain $\dom(A_1)$ is satisfied, 
        and there exists $\varepsilon > 0$ such that the set $\numRange(A_{11}) \cap ([-\varepsilon,0] \times \ui\bbR)$ is bounded.
    \item\label{thm:empty-peripheral-spec:itm:positive} 
        One has $H_1 = L^2(\Omega,\mu)$ for a measure space $(\Omega,\mu)$ and the semigroup $(e^{tA_S})_{t \ge 0}$ on $H_1$ is \emph{positive} in the sense that $0 \le f \in H_1$ implies $e^{tA_S} f \ge 0$ for all $t \ge 0$.%
        \footnote{
            Here, inequalities in $H_1 = L^2(\Omega,\mu)$ are understood in the almost everywhere sense.
        }
    \end{enumerate}
\end{theorem}

\begin{proof}
    \ref{thm:empty-peripheral-spec:itm:compact-embedding}
    Assume that $(e^{tA_S})_{t \ge 0}$ is not exponentially stable. 
    Then it follows from Lemma~\ref{lem:gp-lemma} that there exists a sequence $(\ui\omega_n) \subseteq \ui \bbR$ and a sequence of vectors $(v_n) \subseteq \dom(A_S)$ of norm $\norm{v_n}_{H_1} = 1$ such that $(\ui \omega_n - A_S)v_n \to 0$ in $H_1$.
    
    According to Lemma~\ref{lem:approx-sequence} we have $A_{21}v_n \to 0$ in $H_2$. 
    Hence, the sequence $(v_n)$ is bounded in $\dom(A_{21})$ with respect to the graph norm. 
    As $\dom(A_{21})$ embeds compactly into $H_1$ by assumption, we may assume, after replacing $(v_n)$ and $(\ui\omega_n)$ with appropriate subsequences, that $v_n$ converges to a normalized vector $v \in H_1$ with respect to the $H_1$-norm. 
    Next we show that the sequence $(\ui\omega_n)$ is bounded. 
    Indeed, assume the contrary; after passing to another subsequence we than get $0 \not= \modulus{\omega_n} \to \infty$. 
    It follows from $\ui \omega_n v_n - A_S v_n \to 0$ by dividing by $\omega_n$ that 
    \begin{align*}
        \ui v_n - A_S \big( v_n / \omega_n \big) \to 0.
    \end{align*}
    But $\ui v_n \to \ui v \not= 0$ and $v_n / \omega_n \to 0$ (where both limits take place in $H_1$). 
    This contradicts the closedness of $A_S$.
    Therefore, $(\ui\omega_n)$ is indeed bounded, as claimed. 
    According to Lemma~\ref{lem:approx-sequence-bounded} this contradicts the assumption $\spec(A_\ext) \cap \ui\bbR = \emptyset$.
    
    \ref{thm:empty-peripheral-spec:itm:ev-norm-cont}
    It follows from the spectral assumption of the theorem and from Theorem~\ref{thm:spectrum} that $\spec(A_S) \cap \ui\bbR = \emptyset$. 
    Moreover, due to the eventual norm continuity of $(e^{tA_S})_{t \ge 0}$, the intersection of $\spec(A_S)$ with the vertical strip $[-1,0] \times \ui \bbR$ is compact; see \cite[Theorem~II.4.18]{EngelNagel2000}. 
    Hence the spectral bound of $A_S$ satisfies $\spb(A_S) < 0$. 
    Again due to the eventual norm continuity, the spectral bound $\spb(A_S)$ coincides with the growth bound of $(e^{tA_S})_{t \ge 0}$ \cite[Corollary~IV.3.11]{EngelNagel2000}, which implies the claim.
    
    \ref{thm:empty-peripheral-spec:itm:numerical-range} 
    Assume $(e^{tA_S})_{t\ge0}$ is not exponentially stable. 
    Then it follows from Lemma~\ref{lem:gp-lemma} that there exists a sequence $(\ui\omega_n) \subseteq \ui\bbR$ and a sequence $(v_n) \subseteq \dom(A_S)$ with $\norm{v_n}_{H_1} = 1$ for all $n$ such that
    \begin{align*}
        (\ui \omega_n  - A_S)v_n \to 0.
    \end{align*}
    By testing this convergence against the bounded sequence $(v_n)$ and using that the splitting assumption implies $A_S v_n = (A_{11} + A_{12}SA_{21}) v_n$ for each $n$ (according to Proposition~\ref{prop:splitting-repr}\ref{prop:splitting-repr:itm:A_S}), we thus get
    \begin{align*}
        \inner{v_n}{(\ui \omega_n - A_{11} - A_{12}SA_{21})v_n} \to 0.
    \end{align*}
    Moreover, it follows from Theorem~\ref{thm:skew-hermitian} and Lemma~\ref{lem:approx-sequence} that
    \begin{align*}
        \modulus{\inner{v_n}{A_{12}SA_{21}v_n}} = \modulus{\inner{A_{21} v_n}{S A_{21} v_n}} \le \norm{S} \norm{A_{21}v_n}^2 \to 0
    \end{align*}
    as $n \to \infty$, and thus
    \begin{align*}
        \ui\omega_n - \inner{v_n}{A_{11}v_n}
        = 
        \inner{v_n}{(\ui \omega_n - A_{11} - A_{12}SA_{21})v_n} + \inner{v_n}{A_{12}SA_{21}v_n}
        \to 
        0
    \end{align*}
    as $n \to \infty$.
    Since, for some $\varepsilon > 0$ the set $\numRange(A_{11}) \cap ([-\varepsilon, 0] \times \ui\mathbb R)$ is bounded by assumption, 
    we conclude that $(\omega_n)_{n \in \mathbb N}$ is bounded as well. 
	Thus, Lemma~\ref{lem:approx-sequence-bounded} shows that $\spec(A_\ext) \cap \ui \bbR \neq \emptyset$, which contradicts the assumptions of the theorem.
    
    \ref{thm:empty-peripheral-spec:itm:positive} 
    For positive $C_0$-semigroups on $L^2$-spaces, the spectral bound and the growth bound coincide \cite[Theorem~C-IV-1.1(a)]{Nagel1986}, and the spectral bound is either $-\infty$ or a spectral value \cite[Corollary~C-III-1.4]{Nagel1986}. 
    
    As $0 \not \in \spec(A_\ext)$, it follows from Theorem~\ref{thm:spectrum} that $0 \not\in \spec(A_S)$, so the growth bound of the contractive semigroup $(e^{tA_S})_{t \ge 0}$ is strictly negative.
\end{proof}

\subsection{Strong stability}

In this brief subsection we discuss strong (rather than exponential) stability. 
We call a $C_0$-semigroup $(e^{tA})_{t \ge 0}$ on a Hilbert space $H$ \emph{strongly stable} if $e^{tA}h \to 0$ as $t \to \infty$ for each $h \in H$.

\begin{proposition}
    \label{prop:compact-resolvent-strongly-stable}
    Assume that $A_\ext$ generates a contractive $C_0$-semigroup on $H_1 \times H_2$.
    If $A_S$ has compact resolvent%
    \footnote{
      Note that this is, for instance, satisfied if $A_{21}$ is closed and $\dom(A_{21})$, endowed with the graph norm $\norm{\argument}_{A_{21}}$, embeds compactly into $H_1$.
    }
    and $A_{21}$ is injective, then $(e^{tA_S})_{t \ge 0}$ converges strongly to $0$ as $t \to \infty$.
\end{proposition}

\begin{proof}
    According to Corollary~\ref{cor:non-eigenvalues} the injectivity of $A_{21}$ implies that $A_S$ does not have eigenvalues in $\ui\bbR$. 
    As $A_S$ has compact resolvent, the claim thus follows from the Jacobs--de Leeuw--Glicksberg decomposition theorem for operator semigroups in the strong operator topology \cite[Corollary~V.2.15(i)]{EngelNagel2000} 
    or, alternatively, from the ABLV theorem \cite[Theorem~V.2.21]{EngelNagel2000}.
\end{proof}

\subsection{Long-term behaviour in the parabolic case}
\label{subsec:long-term-parabolic}

If the equivalent conditions of Proposition~\ref{prop:skew-hermitian-equiv} are satisfied and $\dom(A_{21})$, endowed with the graph norm $\norm{\argument}_{A_{21}}$, embeds compactly into $H_1$, one can characterize the long-term behaviour of $(e^{tA_S})_{t \ge 0}$ in terms of the imaginary eigenvalues of $A_{11}$ and the kernel of $A_{21}$:

\begin{theorem} 
    \label{thm:long-time-parabolic-compact}
    In the setting of Section~\ref{section:setting} let the splitting assumption~\ref{ass:splitting} be satisfied and let $A_{11}$ be everywhere defined and bounded. 
    Assume moreover that $A_\ext$ generates a contractive $C_0$-semigroup on $H_1 \times H_2$ 
    and that $\dom(A_{21})$, endowed with the graph norm $\norm{\argument}_{A_{21}}$, embeds compactly into $H_1$.
    \begin{enumerate}[label=\upshape(\alph*)]
        \item\label{thm:long-time-parabolic-compact:itm:zero} 
        The semigroup $(e^{tA_S})_{t \ge 0}$ is exponentially stable if and only if $\ker A_{21} \cap \ker (\ui \omega - A_{11}) = \{0\}$ for each $\omega \in \bbR$;
        \item\label{thm:long-time-parabolic-compact:itm:projection} 
        The semigroup $(e^{tA_S})_{t \ge 0}$ converges with respect to the operator norm to as $t \to \infty$ if and only if $\ker A_{21} \cap \ker (\ui \omega - A_{11}) = \{0\}$ for each $\omega \in \bbR \setminus \{0\}$. 
        In this case the limit operator as $t \to \infty$ is the orthogonal projection onto $\ker A_{21} \cap \ker A_{11}$.
    \end{enumerate}
\end{theorem}

\begin{proof}
    It suffices to prove~\ref{thm:long-time-parabolic-compact:itm:projection}.
    The operator $A_{21}$ is closed according to Proposition~\ref{prop:skew-hermitian-equiv}.
    Endow the domain $\dom(A_S)$ with the graph norm of $A_S$. 
    As $A_S$ generates a $C_0$-semigroup, $A_S$ is closed, so $\dom(A_S)$ is a Banach space. 
    Moreover, $\dom(A_S)$ embeds continuously into $H_1$ is is contained in $\dom(A_S)$ according to Proposition~\ref{prop:skew-hermitian-equiv}. 
    Hence, $A_S$ embeds continuously into $\dom(A_{21})$ by the closed graph theorem. 
    As we assume the embedding $\dom(A_{21}) \hookrightarrow H_1$ to be compact, it follows that $\dom(A_S) \hookrightarrow H_1$ is also compact, so $A_S$ has compact resolvent.

    Moreover, the semigroup $(e^{tA_S})_{t \ge 0}$ is analytic according to Theorem~\ref{thm:via-forms} and hence it is immediately compact. 
    As the semigroup is also contractive, operator norm convergence of $e^{tA_S}$ as $t \to \infty$ is equivalent to assertion that $A_S$ does not have eigenvalues in $\ui \bbR \setminus \{0\}$, see e.g.\ \cite[Corollary~V.3.2 on p.\,330--331]{EngelNagel2000}. 
    The latter property is, in turn, equivalent to the property $\ker A_{21} \cap \ker(\ui \omega - A_{11}) = \{0\}$ for each $\omega \in \bbR \setminus \{0\}$, see Corollary~\ref{cor:splitting-eigenvalues}.

    Finally, assume that $e^{tA_S}$ converges with respect to the operator norm to an operator $P$ on $H_1$ as $t \to \infty$.
    Then $P$ is a projection onto $\ker A_S$, and the latter space is equal to $\ker A_{21} \cap \ker A_{11}$ according to Corollary~\ref{cor:splitting-eigenvalues}. 
    Since each operator $e^{tA_S}$ is contractive so is $P$, and hence the projection $P$ is orthogonal. 
\end{proof}

\section{Coupled wave-heat systems}
\label{sec:wave-heat}

In this section we demonstrate by a simple example how the abstract framework developed in Sections~\ref{section:setting} and~\ref{sec:inheritance-of-stability} can be used to study equations that show both parabolic and hyperbolic behaviour.
We study, on the interval $(0,2)$, the long-term behaviour following of the coupled wave-heat equation 
\begin{equation*}
    \begin{aligned}
        v_{tt} (\zeta,t) &= v_{\zeta\zeta} (\zeta,t) &\quad \text{for } \zeta \in (0,1), t \geq 0, \\
        w_t (\zeta,t) &= \frac{d}{d \zeta} S (\zeta) w_\zeta (\zeta,t) &\quad \text{for } \zeta \in (1,2), t \geq 0, \\
        v_t (0,t) &= 0 &\quad \text{for }  t \geq 0\\
        w (2,t) &= 0 &\quad \text{for } t \geq 0\\
        v_t (1,t) &=  w (1,t) &\quad \text{for } t \geq 0\\
        v_\zeta (1,t) &= S(1) w_\zeta (1,t) &\quad \text{for }  t \geq 0,
    \end{aligned}
\end{equation*}
where $S \in L^\infty(1,2)$ with $\re S \geq \nu$ almost everywhere for some number $\nu > 0$.
It is a well-known phenomena that taking wave-heat equations with appropriate boundary conditions and coupling them at parts of the boundary can lead to a system whose solutions converge strongly but not uniformly to $0$ as $t \to \infty$.
For instance, similar equations as the one above were studied in \cite{BattyPaunonenSeifert2016, Ng2020} for the case $S \equiv 1$. 
Under this assumption those references use concrete computations to estimate the resolvent of the differential operator and thus even obtain polynomial decay rates for the solutions, provided that the initial value is sufficiently smooth. 
A similar phenomenon was also shown to occur on a rectangle in two dimensions \cite{BattyPaunonenSeifert2019}.

We do not prove polynomial decay rates here. 
However, we will demonstrate how our abstract theory can be used to describe the differential operator in our wave-heat equation in an abstract way. 
From the general spectral results in Section~\ref{sec:inheritance-of-stability} we will then very easily obtain strong convergence as $t \to \infty$. 
The point here is that our approach hardly requires explicit computations, which makes it very convenient to handle the presence of the diffusion parameter $S$. 

We first consider the following wave equation on $(0,2)$
\begin{align*}
    u_{tt} (\zeta,t) &= u_{\zeta\zeta} (\zeta,t) &\quad \text{for } \zeta \in (0,2), t \geq 0, \\
    u_t (0,t) &= u_t (2,t) = 0 &\quad \text{for }  t \geq 0.
\end{align*}
By setting $x_1 \coloneqq u_t $ and $x_2 \coloneqq u_\zeta$ we can rewrite this system in port-Hamiltonian form as
\begin{align*}
    \begin{pmatrix}
        \dot x_1 \\ 
        \dot x_2
    \end{pmatrix}
    = 
    \begin{pmatrix}
        0 & \frac{d}{d \zeta} \\ 
        \frac{d}{d\zeta} & 0
    \end{pmatrix}
    \begin{pmatrix}
        x_1 \\ 
        x_2
    \end{pmatrix}    ,
\end{align*}
with state space $L^2(0,2)^2$. 
The domain of the block operator 
$
    \left(
        \begin{smallmatrix}
            0 & \frac{d}{d\zeta} \\ 
            \frac{d}{d\zeta} & 0
        \end{smallmatrix}
    \right)
$ 
is given by $H^1_0(0,2) \times H^1(0,2)$, encoding the boundary conditions of the system.

We restrict the state variables to $(0,1)$ and $(1,2)$ by setting $v_1 \coloneqq x_1 \big\vert_{(0,1)}$, $v_2 \coloneqq x_2\big\vert_{(0,1)}$, $w_1 \coloneqq x_1\big\vert_{(1,2)}$ and $w_2 \coloneqq x_2\big\vert_{(1,2)}$.
Then, the wave equation can be rewritten in these variables as
\begin{align*}
    \begin{pmatrix}
        \dot v_1 \\
        \dot v_2 \\
        \dot w_1 \\
        \dot w_2
    \end{pmatrix}
    =
    \underbrace{\begin{pmatrix}
        0 & \frac{d}{d\zeta} & 0 & 0 \\
        \frac{d}{d\zeta} & 0 & 0 & 0 \\
        0 & 0 & 0 & \frac{d}{d\zeta} \\
        0 & 0 & \frac{d}{d\zeta} & 0 
    \end{pmatrix}}_{=:A_\ext}
    \begin{pmatrix}
        v_1 \\
        v_2 \\
        w_1 \\
        w_2 
    \end{pmatrix}.
\end{align*}
Let  
\begin{align*}
  H_1 &:= L^2(0,1)^2 \times L^2(1,2)\\
  H_2 &:=  L^2(1,2) .   
\end{align*}
The operator $A_\ext: \dom(A_\ext)\subseteq H_1\times H_2 \rightarrow H_1\times H_2$ is equipped with the domain  
\begin{align*}
    \dom(A_\ext) 
    & :=
    \Biggl\{ 
        \begin{psmallmatrix}
            v_1 \\
            v_2 \\
            w_1 \\    
            w_2 \\
        \end{psmallmatrix} 
        \Bigg\vert \, 
        \exists x_1 \in H^1_0(0,2), \; x_2 \in H^1(0,2) \text{ such that } \\
        & \qquad \qquad 
        v_1 = x_1\big\vert_{(0,1)}, \; w_1 = x_1\big\vert_{(1,2)}, \;
        v_2 = x_2\big\vert_{(0,1)}, \; w_2 = x_2\big\vert_{(1,2)} 
    \Biggl\}
    \\
    & = 
    \Biggl\{ 
        \begin{psmallmatrix}
            v_1 \\
            v_2 \\
            w_1 \\    
            w_2 \\
        \end{psmallmatrix} 
        \Bigg\vert \, 
        v_1, v_2 \in H^1(0,1), \; w_1, w_2 \in H^1(1,2) \text{ and } \\ 
        & \qquad \qquad 
        v_1(0) = w_1(2) = 0, \; 
        v_1(1) = w_1(1), \; v_2(1) = w_2(1)
    \Biggl\}.
\end{align*}

Furthermore we define
$A_1:\dom(A_{1}) \subseteq H_1\times H_2  \rightarrow H_1$ by 
\begin{align*}
    A_1 \coloneqq 
    \begin{pmatrix}
        0 & \frac{d}{d\zeta} & 0 & 0 \\
        \frac{d}{d\zeta} & 0 & 0 & 0 \\
        0 & 0 & 0 & \frac{d}{d\zeta} \\
    \end{pmatrix}, \quad
    \dom(A_1) = \dom (A_\ext),
\end{align*}
and $A_{21}: \dom(A_{21}) \subseteq H_1 \to H_2$ by
\begin{align*}
    A_{21} &\coloneqq 
    \begin{pmatrix}
        0 & 0 & \frac{d}{d\zeta} \\
    \end{pmatrix}, \\
    \dom (A_{21}) 
    &= 
    L^2(0,1)^2 \times \{ w_1 \in L^2(1,2) \suchthatManual \exists x_1 \in H^1_0(0,2) \text{ s.t. } x_1\big\vert_{(1,2)} = w_1 \} 
    \\ 
    &= 
    L^2(0,1)^2 \times \{ w_1 \in H^1(1,2) \suchthatManual w_1(2) = 0 \}
    .
\end{align*}

From now on we identify the multiplication operator induced by $S$ with $S$. Hence, the operator $S: H_2 \to H_2$ is the bounded linear operator which is coercive from Section \ref{section:setting}. We obtain the operator $A_S$ by the construction in Definition \ref{def:operators} \ref{def:operators:itm:A_S} as
\begin{align*}
    A_S
    \begin{pmatrix}
        v_1 \\
        v_2 \\
        w_1 \\
    \end{pmatrix}
    =
    A_1
    \begin{pmatrix}
        v_1 \\
        v_2 \\
        w_1 \\
        S \frac{d}{d\zeta} w_1
    \end{pmatrix}
    =
    \begin{pmatrix}
        0 & \frac{d}{d\zeta} & 0 \\
        \frac{d}{d\zeta} & 0 & 0 \\
        0 & 0 & \frac{d}{d\zeta} S \frac{d}{d\zeta} \\
    \end{pmatrix}
    \begin{pmatrix}
        v_1 \\
        v_2 \\
        w_1 \\
    \end{pmatrix}
    .
\end{align*}
The domain of this operator is given by
\begin{align*}
    \dom(A_S) &= 
    \left \{ 
        \begin{pmatrix}
            v_1 \\
            v_2 \\
            w_1 \\
        \end{pmatrix} 
        \in \dom(A_{21})
        \suchthat
        \begin{pmatrix}
            v_1 \\
            v_2 \\
            w_1 \\
            S \frac{d}{d\zeta} w_1 
        \end{pmatrix}
        \in \dom(A_1)
    \right \} \\
    &=
    \Biggl\{
        \begin{pmatrix}
            v_1 \\
            v_2 \\
            w_1 \\
        \end{pmatrix}
        \Bigg| \,
        \exists x_1 \in H^1_0(0,2), x_2 \in H^1(0,2) \text{ such that } \\
        & \qquad \qquad  
        v_1 = x_1\big\vert_{(0,1)}, \; w_1 = x_1\big\vert_{(1,2)}, \; 
        v_2 = x_2\big\vert_{(0,1)}, \; S \frac{d}{d\zeta} w_1 = x_2\big\vert_{(1,2)} 
    \Biggr\}
    \\ 
    &= 
    \Biggl\{
        \begin{pmatrix}
            v_1 \\
            v_2 \\
            w_1 \\
        \end{pmatrix}
        \Bigg| \,
        v_1, w_1 \in H^1(0,1), \; w_1 \in H^1(1,2), \; S \frac{d}{d\zeta} w_1 \in H^1(1,2) \text{ and } \\
        & \qquad \qquad
        v_1(0) = w_1(2) = 0, \; 
        v_1(1) = w_1(1), \; 
        v_2(1) = \Big(S \frac{d}{d\zeta} w_1\Big)(1)
    \Biggr\}
    .
\end{align*}
From Theorem \ref{thm:inheritance-of-domination} and since the operator $A_{\ext}$ of the wave equation generates a contraction semigroup on $L^2(0,2)^2$, see \cite{JaSk21}, $A_S$ generates a contraction semigroup on $H_1$. 

\begin{theorem}
    The operator $A_S: \dom(A_S) \subseteq L^2(0,1)^2 \times L^2(1,2) \to  L^2(0,1)^2 \times L^2(1,2)$ generates a strongly stable semigroup.
\end{theorem}

\begin{proof}
    From Proposition~\ref{prop: compact-resolvent} and since $H^1(0,2)$ is compactly embedded into $L^2(0,2)$, we know that the operator $A_S$ has compact resolvent.
    Hence, it suffices to show that $A_S$ has no eigenvalues in $ \ui\bbR$; the strong stability of the semigroup generated by $A_S$ then follows from \cite[Corollary~V.2.15(i)]{EngelNagel2000} or from the ABLV theorem \cite[Theorem~V.2.21]{EngelNagel2000}, as in the proof of Proposition \ref{prop:compact-resolvent-strongly-stable}.

    Let $\omega \in \bbR$, and assume $i\omega$ is an eigenvalue of $A_S$ with eigenvector
    $
        \begin{psmallmatrix}
            v_1 \\ v_2 \\ w_1
        \end{psmallmatrix}
    $.
    By Theorem \ref{thm:eigenvalues} this is equivalent to $\begin{psmallmatrix}
        v_1 \\ v_2 \\ w_1 \\ 0
    \end{psmallmatrix} \in \dom(A_\ext)$ and
    \begin{align*}
        i \omega 
        \begin{pmatrix}
            v_1 \\ v_2 \\ w_1 \\ 0
        \end{pmatrix}
        =
        A_\ext \begin{pmatrix}
            v_1 \\ v_2 \\ w_1 \\ 0
        \end{pmatrix}
        =
        \begin{pmatrix}
            \frac{d}{d\zeta} v_2 \\
            \frac{d}{d\zeta} v_1 \\
            0 \\
            \frac{d}{d\zeta} w_1
        \end{pmatrix}.
    \end{align*}
    Since 
    $
        \begin{psmallmatrix}
            v_1 \\ 
            v_2 \\ 
            w_1 \\
            0
        \end{psmallmatrix} 
        \in \dom(A_\ext)
    $ 
    one has $v_1,v_2 \in H^1(0,1)$ and $w_1 \in H^1(1,2)$, as well as the boundary conditions $v_1(0) = w_1(2) = 0$, $v_1(1) = w_1(1)$, and $v_2(1) = 0$.
   
    For $\omega = 0$ this implies that $v_1$, $v_2$, and $w_1$ are constant and hence equal to $0$.
    For $\omega \in \bbR \setminus \{0\}$ on the other hand, we obtain $w_1 = 0$ as well as that $i\omega$ is an eigenvalue of 
    $
        \begin{psmallmatrix}
            0 & \frac{d}{d\zeta} \\
            \frac{d}{d\zeta} & 0
        \end{psmallmatrix}
    $ 
    with eigenvector 
    $
        \begin{psmallmatrix}
            v_1 \\ 
            v_2
        \end{psmallmatrix}
    $
    that satisfies the boundary condition $v_1(0)=v_1(1) = 0$ and $v_2(1) = 0$.
    From this one can derive by a short computation that $v_1 = v_2 = 0$.
    Thus, in both cases we obtain 
    $\begin{psmallmatrix}
        v_1 \\
        v_2\\
        w_1 \\
        0
    \end{psmallmatrix} =0$, which contradicts $i\omega$ being an eigenvalue of $A_\ext$.
    Hence $A_S$ cannot have an imaginary eigenvalue either. 
\end{proof}

\section{Parabolic equations coupled by matrix-valued potentials}
\label{sec:coupled-parabolic-equations}

The fact that Corollary~\ref{cor:splitting-eigenvalues} prevents the occurrence of Turing instability (see the discussion after Example~\ref{exa:turing-instab}) can be used to show convergence to equilibrium as $t \to \infty$ for various classes of parabolic systems that are coupled by a matrix-valued potential.
In Subsection~\ref{subsec:matrix-coupling-general} we provide a general framework for the analysis of such coupled equations. 
In Subsection~\ref{subsec:matrix-coupling-heat} we show that this general framework incorporates a result about coupled heat equations from \cite[Thm.~3.7]{DoGl22} as a special case and in Subsection~\ref{subsec:matrix-coupling-biharmonic} we demonstrate how one can, in the same vein, study the long-term behaviour of a coupled system of biharmonic parabolic equations.

\subsection{General setting}
\label{subsec:matrix-coupling-general}

Throughout this subsection let $(\Omega,\mu)$ be a measure space, set $H_1 \coloneqq L^2(\Omega,\mu)$, let $H_2$ be a Hilbert space, let $A_{21}: \dom(A_{21}) \subseteq H_1 \to H_2$ be a closed and densely defined linear operator, and let $A_{12} \coloneqq -A_{21}^*: \dom\big(A_{21}^*\big) \subseteq H_2 \to H_1$ be its negative adjoint. 
We set 
$
    A_1 
    \coloneqq 
    \begin{pmatrix} 
        0 & A_{12}
    \end{pmatrix}
$
on the domain $H_1 \times \dom\big(A_{21}^*\big)$.
Then $A_1$ satisfies the splitting assumption~\ref{ass:splitting} (where $A_{11}$ is the zero operator on $H_1$) and the equivalent assertions of Proposition~\ref{prop:skew-hermitian-equiv} are satisfied. 

In contrast to Section~\ref{section:setting} we do not only consider a single operator $S$ on $H_1$ now. 
Instead, we fix an integer $N \ge 1$ and for each $k \in \{1, \dots, N\}$ consider a bounded linear operator $S_k: H_1 \to H_1$ that satisfies $\re \inner{S_k h_1}{h_1} \ge \nu \norm{h_1}_{H_1}^2$ for an ($k$-independent) number $\nu > 0$ and all $h_1 \in H_1$. 
Thus we get $N$ operators $A_{S_k}: \dom(A_{S_k}) \subseteq H_1 \to H_1$. 
According to Proposition~\ref{prop:skew-hermitian-equiv} each operator $A_{S_k}$ is given by 
\begin{align*}
    \dom(A_{S_k}) 
    & = 
    \big\{
        h_1 \in \dom(A_{21}) \suchthatManual S_k A_{21} h_1 \in \dom(A_{21}^*)
    \big\}
    \qquad \text{and}
    \\
    A_{S_k} h_1 
    & = 
    -A_{21}^* S_k A_{21} h_1 
\end{align*}
for all $h_1 \in \dom(A_{S_k})$.
Each operator $A_{S_k}$ generates a contractive and analytic $C_0$-semigroup on $H_1$ (Theorems~\ref{thm:inheritance-of-domination} and~\ref{thm:via-forms}).  
Moreover, Theorem~\ref{thm:via-forms} tells us that $-A_{S_k}$ is associated to the form $a_k: \dom(A_{21}) \times \dom(A_{21}) \to \bbC$ given by 
\begin{align*}
    a_k(u,v) = \inner{S_k A_{21}u}{A_{21}v}.
\end{align*}
for all $u,v \in \dom(A_{21})$.

We now couple the operators $A_{S_k}$ by adding a matrix-valued potential. 
To this end, let $V: \Omega \to \bbC^{N \times N}$ be measurable and (essentially) bounded. 
Then $V$ defines a bounded linear operator on $H_1^N = L^2(\Omega,\mu)^N \simeq L^2(\Omega,\mu;\bbC^N)$ via the mapping $u \mapsto Vu$; we denote this operator also by the symbol $V$.
The coupling of the operators $A_{S_k}$ via $V$ is modelled by the operator $C: \dom(A_{S_1}) \times \dots \times \dom(A_{S_N}) \subseteq H_1^N \to H_1^N$ that is given by
\begin{align}
    \label{eq:op-coupled-by-matrix-potential}
    C 
    \coloneqq 
    \begin{pmatrix}
        A_{S_1} &        &     \\ 
                & \hdots &     \\ 
                &        & A_{S_N}
    \end{pmatrix}
    + 
    V
\end{align}
The operator $C$ can equivalently be described by a sesqui-linear form or by the formalism from Section~\ref{section:setting}:

\begin{proposition}
    \label{prop:coupled-representations-of-A_S}
    Under the assumptions specified above the following holds: 
    \begin{enumerate}[label=\upshape(\alph*)]
        \item\label{prop:coupled-representations-of-A_S:itm:form} 
        The operator $-C$ is associated to the sesquilinear form $\big(a, \dom(A_{21})^N\big)$ on $H_1^N$ that is given by 
        \begin{align*}
            a(u,v) = a_1(u_1,v_1) + \dots + a_N(u_N,v_N)
            - \inner{Vu}{u}
        \end{align*}
        for all $u,v \in \dom(A_{21})^N$.

        \item\label{prop:coupled-representations-of-A_S:itm:A_ext} 
        The operator $C$ can be represented as an operator $A_S$ as defined in Section~\ref{section:setting}, where $S$ and $A_\ext$ are given by
        \begin{align*}
            S 
            = 
            \begin{pmatrix}
                S_1 &        &     \\ 
                    & \ddots &     \\ 
                    &        & S_N 
            \end{pmatrix}
        \end{align*}
        and
        \begin{align*}
            A_\ext 
            = 
            \left(
                \begin{array}{ccc|ccc}
                           &        &        & A_{12} &        &         \\
                           &   V    &        &        & \ddots &         \\
                           &        &        &        &        &  A_{12} \\
                    \hline 
                    A_{21} &        &        &        &        &         \\
                           & \ddots &        &        &        &         \\
                           &        & A_{21} &        &        &   
                \end{array}
            \right)
            ,
        \end{align*}
        with the natural domain for $A_\ext$.
    \end{enumerate}
\end{proposition}

The proof of the proposition is straighforward, so we omit it.
Part~\ref{prop:coupled-representations-of-A_S:itm:A_ext} of the proposition along with our abstract spectral and asymptotic results from Section~\ref{sec:inheritance-of-stability} now enable us to analyze the imaginary eigenvalues of the operator $C$ given in~\eqref{eq:op-coupled-by-matrix-potential} and the long-term behaviour of the semigroup $(e^{tC})_{t \ge 0}$.

\begin{theorem}
    \label{thm:long-term-coupled-by-matrix-potential-general}
    Let $w \in H_1 = L^2(\Omega,\mu)$ be a function that is non-zero almost everywhere and assume that the kernel $\ker A_{21}$ is spanned by $w$.
    Assume that the matrix $V(\theta) \in \bbC^{N \times N}$ is dissipative for almost all $\theta \in \Omega$.
    Then the following are equivalent for the coupled operator $C$ given in~\eqref{eq:op-coupled-by-matrix-potential}: 
    \begin{enumerate}[label=\upshape(\roman*)]
        \item\label{thm:long-term-coupled-by-matrix-potential-general:itm:ev-C} 
        The operator $C$ does not have any eigenvalues in $\ui \bbR \setminus \{0\}$.

        \item\label{thm:long-term-coupled-by-matrix-potential-general:itm:ev-V} 
        The eigenspaces of the matrices $V(\theta) \in \bbC^{N \times N}$ (where $\theta \in \Omega$) satisfy the following: 
        for every $\ui \beta \in \ui \bbR \setminus \set{0}$ and every measurable subset $\widetilde{\Omega} \subseteq \Omega$ with full measure one has
        \begin{equation*}
            \bigcap_{\theta \in \widetilde{\Omega}} \ker (\ui \beta - V(\theta)) = \set{0}.
        \end{equation*}
    \end{enumerate}
    If, in addition, the domain $\dom(A_{21})$, endowed with the graph norm of $A_{21}$, embeds compactly into $H_1$, then the above assertions~\ref{thm:long-term-coupled-by-matrix-potential-general:itm:ev-C} and~\ref{thm:long-term-coupled-by-matrix-potential-general:itm:ev-V} are also equivalent to:
    \begin{enumerate}[resume, label=\upshape(\roman*)]
        \item\label{thm:long-term-coupled-by-matrix-potential-general:itm:long-term} 
        The semigroup $(e^{tC})_{t \ge 0}$ converges with respect to the operator norm as $t \to \infty$.
    \end{enumerate}
\end{theorem}

\begin{proof}
    \equivalentProof{thm:long-term-coupled-by-matrix-potential-general:itm:ev-C}{thm:long-term-coupled-by-matrix-potential-general:itm:ev-V}
    Fix $\ui \beta \in \ui \bbR$. 
    Assertion~\ref{thm:long-term-coupled-by-matrix-potential-general:itm:ev-C} is, according to Corollary~\ref{cor:splitting-eigenvalues}, equivalent to the subspace 
    \begin{align*}
        \ker 
        \begin{pmatrix}
            A_{21} &        &        \\ 
                   & \ddots &        \\
                   &        & A_{21} 
        \end{pmatrix}
        \cap 
        \ker (\ui \beta - V)
    \end{align*}
    being zero. 
    Since the kernel of $A_{21}$ is, by assumption, spanned by $w$ the above subspace equals 
    \begin{align*}
        \bbC^N w
        \, \cap \, 
        \ker (\ui \beta - V)
        .
    \end{align*}
    As $w$ is, by assumption, non-zero almost everywhere 
    this space is zero if and only if $\bigcap_{\theta \in \widetilde{\Omega}} \ker (\ui \beta - V(\theta)) = \set{0}$ for every measurable subset $\tilde \Omega \subseteq \Omega$ of full measure. 

    \equivalentProof{thm:long-term-coupled-by-matrix-potential-general:itm:ev-V}{thm:long-term-coupled-by-matrix-potential-general:itm:long-term} 
    Assume now that $\dom(A_{21})$ with the graph norm embedes compactly into $H_1$. 
    Then the operator $C$ has compact resolvent. 
    As the semigroup generated by $C$ is analytic according to Theorem~\ref{thm:via-forms} it is, hence, immediately compact \cite[Theorem~II.4.29 on p.\,119]{EngelNagel2000}. 
    So the equivalence of~\ref{thm:long-term-coupled-by-matrix-potential-general:itm:ev-V} and~\ref{thm:long-term-coupled-by-matrix-potential-general:itm:long-term} follows from the boundedness of the semigroup and the general asymptotic theory of immediately compact semigroups \cite[Corollary~V.3.2 on pp.\,330--331]{EngelNagel2000}.
\end{proof}

\subsection{Coupled heat equations}
\label{subsec:matrix-coupling-heat}

In this section we study a system of heat equations on a bounded domain in $\bbR^d$, subject to Neumann boundary conditions, which is coupled by a matrix-valued potential. 
We show that a recent result from \cite{DoGl22} about the long-time behaviour of such equations is an easy special case of the general theory in the previous Subsection~\ref{subsec:matrix-coupling-general}.

Let $\emptyset \neq \Omega \subseteq \bbR^d$ be a bounded domain which has a \textit{Sobolev extension property}, that is, every Sobolev function in $H^1(\Omega;\bbC)$ is the restriction of a Sobolev function in $H^1(\bbR^d;\bbC)$. 
Furthermore let $N \geq 1$ and $S_1, \dots, S_N: \Omega \to \bbR^{n \times n}$ and $V: \Omega \to \bbC^{N \times N}$ be bounded and measurable functions.
We assume that there exists a constant $\nu > 0$ such that for all $k \in \left\{ 1, \dots, N \right\}$ and almost all $\theta \in \Omega$, the \textit{uniform ellipticity condition}
\begin{equation*}
    \re(\bar{\xi}^T S_k(\theta) \xi) \geq \nu \norm{\xi}_2
\end{equation*}
holds for all $\xi \in \bbC^d$.
Each of the matrix-valued functions $S_k$ can be identified with a multiplication operator on $L^2(\Omega;\bbC^d)$ which we also denote by $S_k$.
Our goal is to study the long-term behaviour of the heat solutions to the couple system of heat equations
\begin{equation}
    \label{eq:coupled_parabolic}
    \frac{\dx}{\dx t} \begin{pmatrix}
        x_1 \\
        \vdots \\ 
        x_N \\
    \end{pmatrix} 
    = 
    \begin{pmatrix}
        \div(S_1 \grad x_1) \\
        \vdots \\
        \div(S_N \grad x_N) \\
    \end{pmatrix} 
    + 
    V(\theta) 
    \begin{pmatrix}
        x_1 \\
        \vdots \\ 
        x_N \\
    \end{pmatrix},
\end{equation}
for $t \ge 0$ and subject to Neumann boundary conditions.
To give precise meaning to the expression ``$x \mapsto \div(S_k \grad x)$" one can use form methods and thus end up in the setting of Subsection~\ref{subsec:matrix-coupling-general}.
To do so we define, for each $k \in \left\{ 1, \dots, N \right\}$, the form
\begin{align*}
    a_k: H^1(\Omega;\bbC) \times H^1(\Omega; \bbC) &\to \bbC
\end{align*}
by 
\begin{align*}
    a_k(x,z) \coloneqq \inner{S_k\grad x}{\grad z}_{L^2(\Omega; \bbC^n)}
\end{align*}
for all $x,z \in H^1(\Omega;\bbC)$; 
we denote the associated operators on $L^2(\Omega)$ by $-A_{S_k} : \dom(A_{S_k}) \subseteq L^2(\Omega;\bbC) \to L^2(\Omega; \bbC)$. 
Those operators are commonly interpreted as a realization of the divergence form operators $x \mapsto - \div S_k \grad x$ with Neumann boundary conditions.

One can also represent the operators $A_{S_k}$ within the framework used in this paper:
as explained in Subsection~\ref{subsec:matrix-coupling-general} each operator $A_{S_k}$ is then given by 
\begin{align*}
    \dom(A_{S_k}) 
    & = 
    \big\{ 
        x \in H^1(\Omega;\bbC) 
        \suchthatManual
        S_k \dom(\grad^*) x
    \big\}, \\ 
    A_{S_k} x 
    & = 
    - \grad^* S_k \grad x
    ,
\end{align*}
where the gradient $\grad$ is considered as an operator $H^1(\Omega;\bbC) \subseteq L^2(\Omega;\bbC) \to L^2(\Omega;\bbC^d)$ and is the operator $A_{21}$ from Section~\ref{subsec:matrix-coupling-general}.  
If $\Omega$ has sufficiently smooth boundary one can check that the dual operator $\grad^*: \dom(\grad^*) \subseteq L^2(\Omega;\bbC^d) \to L^2(\Omega;\bbC)$ acts as $-\div$ and that its domain $\dom(\grad^*)$ consists of those vector fields in $L^2(\Omega; \bbC^d)$ whose distributional divergence is in $L^2(\Omega;\bbC)$ and that vanish, in a weak sense, in normal direction on $\partial \Omega$. 
This explains, again, that each $A_{S_k}$ is a realization of $x \mapsto - \div S_k \grad x$ with Neumann boundary conditions.

The coupled parabolic equation \eqref{eq:coupled_parabolic} can now be stated as the abstract Cauchy problem 
\begin{equation*}
    \frac{dx}{dt}= \underbrace{\begin{pmatrix}
        A_{S_1} & &  \\
        & \ddots & \\
         & & A_{S_N}
    \end{pmatrix}}_{\eqqcolon \calB_2}x + V x
\end{equation*}
on $L^2(\Omega; \bbC^N)$ for $t \ge 0$.
Let us characterize under which conditions all solutions to this equation converge as $t \to \infty$:

\begin{theorem}\label{thm:convergence_of_solutions}
    Let the matrix $V(\theta)$ be dissipative on $\bbC^N$ (endowed with the $\ell^2$-norm) for almost all $\theta \in \Omega$. Then the following assertions are equivalent:
    \begin{enumerate}[label=\upshape(\roman*)]
        \item 
        The solutions to the coupled heat equation \eqref{eq:coupled_parabolic} converge (uniformly on the unit ball) as $t \to \infty$.
        
        \item 
        For every $\ui \beta \in \ui \bbR \setminus \set{0}$ and every measurable subset $\widetilde{\Omega} \subseteq \Omega$ with full measure we have
        $
            \bigcap_{\theta \in \widetilde{\Omega}} \ker (\ui \beta - V(\theta)) = \set{0}.
        $
    \end{enumerate}
\end{theorem}

\begin{proof}
    This follows from Theorem~\ref{thm:long-term-coupled-by-matrix-potential-general} if we observe two things: 
    The kernel of $A_{21} = \grad: H^1(\Omega;\bbC) \to L^2(\Omega;\bbC^d)$ is spanned by the constant function $\one$ since $\Omega$ is a domain and thus connected. 
    And the embedding of $H^1(\Omega)$ (endowed with the $H^1(\Omega)$-norm, which equals the graph norm of the gradient operator) into $L^2(\Omega)$ is compact since $\Omega$ is assumed to have the extension property.
\end{proof}

Theorem~\ref{thm:convergence_of_solutions} can also be found in \cite[Thm.~3.7]{DoGl22}. 
A close look at the proof in this reference reveals that the key arguments are similar to the arguments that we used in the alternative proof of Corollary~\ref{cor:splitting-eigenvalues} and in the proof of Theorem~\ref{thm:long-term-coupled-by-matrix-potential-general}.
These arguments and the framework presented in Subsection~\ref{subsec:matrix-coupling-general} are an abstraction of the situation discussed in \cite[Thm.~3.7]{DoGl22} and, as we have just seen, immediately give this result as a special case. 
The advantage of our abstract approach is that it easily yields results on the long-term behaviour of other systems of differential equations that are coupled by a matrix-valued potential. 
We demonstrate this with one further class of equations in the next subsection.

\subsection{Coupled biharmonic heat equations}
\label{subsec:matrix-coupling-biharmonic}

As a second special case of coupled parabolic equations we consider the bi-harmonic equations with Neumann boundary conditions on a bounded domain in $\bbR^d$. 
More precisely, we make the following assumptions: 

Let $\Omega \subseteq \bbR^d$ be a bounded domain which we assume, as in  Subsection~\ref{subsec:matrix-coupling-heat}, to have the extension property. 
As before, fix an integer $N \ge 1$. 
For each $k \in \{1, \dots, N\}$ let $s_k: \Omega \to \bbC$ a measurable and bounded scalar-valued function and assume that $\re s_k(\theta) \ge \nu$ for a number $\nu > 0$, almost all $\theta \in \Omega$, and all $k \in \{1, \dots, N\}$. 
Let $\Delta: \dom(\Delta) \subseteq L^2(\Omega) \to L^2(\Omega)$ denote the Neumann Laplace operator. 
As $\Omega$ is not assumed to have smooth boundary, one can give meaning to the notion \emph{Neumann Laplace} by the same means as in Subsection~\ref{subsec:matrix-coupling-heat}. 

As in the previous two subsections, let $V: \Omega \to \bbC^{N \times N}$ be measurable and bounded. 
We want to describe the long-term behaviour of the solutions $x: [0,\infty) \to L^2(\Omega; \bbC^N) \simeq L^2(\Omega;\bbC)^N$ to the coupled biharmonic heat equation 
\begin{align}
    \label{eq:coupled_biharmonic}
    \begin{pmatrix}
        \dot x_1(t) \\ 
        \vdots \\ 
        \dot x_N(t)
    \end{pmatrix}
    = 
    \begin{pmatrix}
        \Delta \big(s_1 \Delta x_1(t)\big) \\ 
        \vdots \\ 
        \Delta \big(s_N \Delta x_N(t)\big)
    \end{pmatrix}
    + 
    V 
    \begin{pmatrix}
        x_1(t) \\ 
        \dots \\ 
        x_N(t)
    \end{pmatrix}
\end{align}
for $t \ge 0$.
As we assumed $\Delta$ to be the Neumann Laplace operator, this evolution equation is subject to the Neumann type boundary conditions
\begin{align*}
    \frac{\partial}{\partial \nu} x_k(t) 
    = 
    \frac{\partial}{\partial \nu} \big( s_k x_k(t) \big) = 0
\end{align*}
on $\partial \Omega$ for each $k \in \{1, \dots, N\}$, where $\frac{\partial}{\partial \nu}$ denotes the (outer) normal derivative.

To employ the machinery from Subsection~\ref{subsec:matrix-coupling-general}, we let $S_k: L^2(\Omega;\bbC) \to L^2(\Omega; \bbC)$ denote the multiplication with $s_k$ for each $k \in \{1, \dots, N\}$, we set $A_{12} \coloneqq - \Delta$ as well as $A_{21} \coloneqq \Delta$ and 
$
    A_1 
    \coloneqq 
    \begin{pmatrix} 
        0 & A_{12}
    \end{pmatrix}
    =
    \begin{pmatrix} 
        0 & \Delta
    \end{pmatrix}
$.
Then the general Theorem~\ref{thm:long-term-coupled-by-matrix-potential-general} gives the same characterization of convergence of the solutions to~\eqref{eq:coupled_biharmonic} that we also obtained for the coupled heat equation in Subsection~\ref{subsec:matrix-coupling-heat}. 
For the convenience of the reader and the sake of completeness we state the complete theorem.
\begin{theorem}\label{thm:convergence_of_solutions-2}
    Let the matrix $V(\theta)$ be dissipative on $\bbC^N$ (endowed with the $\ell^2$-norm) for almost all $\theta \in \Omega$. Then the following assertions are equivalent:
    \begin{enumerate}[label=\upshape(\roman*)]
        \item 
        The solutions to the coupled biharmonic heat equation \eqref{eq:coupled_biharmonic} converge (uniformly on the unit ball) as $t \to \infty$.
        
        \item 
        For every $\ui \beta \in \ui \bbR \setminus \set{0}$ and every measurable subset $\widetilde{\Omega} \subseteq \Omega$ with full measure we have
        $
            \bigcap_{\theta \in \widetilde{\Omega}} \ker (\ui \beta - V(\theta)) = \set{0}.
        $
    \end{enumerate}
\end{theorem}

\begin{proof}
    We can apply Theorem~\ref{thm:long-term-coupled-by-matrix-potential-general} if we observe the following two things: 
    The kernel of the Neumann Laplace operator $\Delta = A_{21}$ is spanned by the constant function $\one$ since $\Omega$ is a domain and thus connected. 
    And the embedding of $\dom(A_{21})$, endowed with the graph norm, into $L^2(\Omega)$ is compact as $\Omega$ has the extension property; indeed, the extension property implies that the semigroup generated by $A_{21}$ is ultracontractive \cite[Subsections~7.3.2 and~7.3.2]{Ar04} and hence immediately compact, so $A_{21}$ has compact resolvent \cite[Theorem~II.4.29 on p.\,119]{EngelNagel2000}.
\end{proof}

\section{Stability of one-dimensional port-Hamiltonian systems with dissipation}
\label{sec:phs}

Port-based network modelling of complex physical systems leads to port-Hamil\-tonian systems. For 
finite-dimensional systems there is by now a well-established theory
\cite{vanDerSchaft06,EbMS07,DuinMacc09}. 
This has been extended to the infinite-dimensional situation by a functional analytic approach that was used to study 
well-posedness, see e.g.\ \cite{GoZwMa05, Sk21, ZwGoMaVi10}, stability and stabilization \cite{RaGoMaZw14,AuJa14,Au16} and controller design \cite{JaKa19,HuPa}. 
For an overview we refer, for instance, to the book \cite{JaZw12}
and the survey article \cite{JaZw19}.

A particular focus in the literature so far is laid on linear infinite-dimensional port-Hamiltonian systems on a one-dimensional spatial domain. 
These systems can be characterized in a geometric way by a so-called Dirac structure together with a Hamiltonian and
dissipation can then be modelled by adding a resistive relation; 
we refer for instance to \cite{Re21} for a detailed explanation of this approach. 
In particular, non isothermal reaction diffusion processes \cite{ZhHaGo15}, dissipative shallow water equations 
\cite{CaMaLe21} and coupled heat transfer \cite{JaEhGuJa22} can be modelled as infinite-dimensional port-Hamiltonian systems with dissipation.

In this section we consider evolution equations 
\begin{align}
    \label{eq:introduction-L}
    \dot z(t) = L z(t) 
    \qquad \text{for } t \geq 0
\end{align}
on the space $L^2\big( (0,1); \bbC^n \big)$, where $L$ is a second-order differential operator with port-Hamiltonian structure and is subject to a general set of boundary conditions; see Subsection~\ref{subsec:diff-eq} for details. 
We will use the abstract setting from Section~\ref{sec:inheritance-of-stability} to give sufficient conditions for exponential stability of such systems in Subsections~\ref{subsec:well-posedness-and-stability} and~\ref{subsec:heat-equation-nl}.
Hence, this section contributes to the stability theory of port-Hamiltonian systems. This is a very active topic as evidenced, in addition to the aforementioned references, for instance by the contributions \cite{RaZwGo17,ScZw18,Sc22,WaurickZwart2022}.

We note that for first order port-Hamiltonian systems -- i.e., where only spatial derivatives up to order $1$ occur and dissipation can only occur at the boundary -- exponential stability of the solutions is well-understood, and easily verifiable sufficient conditions are available, see for instance \cite[Chapter~9]{JaZw12}.

\subsection{The evolution equation}
\label{subsec:diff-eq}

In this section we apply the results of Section~\ref{sec:inheritance-of-stability} to a class of $\bbC^n$-valued one-dimensional port-Hamiltonian systems. 
We are looking for a function $z: [0,\infty) \to L^2\big((0,1);\bbC^n\big)$ which satisfies the evolution equation 
\begin{align}
    \label{eq:dissipative-phs}
    \dot z(t) 
    = &
    \Big(
        P_1 \partial + P_0 
        \; + \; 
        \big(G_1 \partial + G_0\big) \, S \, \big(G_1^* \partial - G_0^* \big) 
    \Big) 
    \,
    \calH \, z(t) , 
\end{align}
for all $t \ge 0$ -- where $\partial$ denotes the spatial derivative -- together with the initial condition 
\begin{align}
    \label{eq:inital-condition}
    z(0) & =  z_0
\end{align}
for an initial function $z_0 \in L^2\big((0,1); \bbC^n\big)$, 
and subject to the following boundary conditions for $x \coloneqq z(t)$ at every time $t \ge 0$:
\begin{align} 
    \label{eq:boundary-conditions}
    0 & = 
    \tilde W_\bdd
    \begin{pmatrix}
        \calH x\, (1)  \\
        S \big( G_1^* \partial - G_0^*\big) \calH \, x (1)  \\
        \calH x \, (0)  \\
        S \big( G_1^* \partial - G_0^*\big) \calH \, x (0)
    \end{pmatrix}
    .
\end{align}
As in the previous sections we will study mild solutions in the sense of $C_0$-semigroup theory.
To this end, we will formulate the equations~\eqref{eq:dissipative-phs}--\eqref{eq:boundary-conditions} in operator theoretic language as an abstract Cauchy problem at the end of Subsection~\ref{subsec:diff-eq}. 
The boundary conditions will be encoded in the domain of the corresponding differential operator, which will also make it clear how to interpret expressions such as $\calH x(1)$ that occur in~\eqref{eq:boundary-conditions}.

Note that, while the boundary conditions~\eqref{eq:boundary-conditions} might look slightly intimidating at first glance, they are a natural way to take into account both the values and the spatial derivatives of $x$ on the boundary in a way that reflects the structure of the differential operator in~\eqref{eq:dissipative-phs}.
Moreover, we will see later on that this formulation of the boundary conditions makes it very convenient to formulate results about well-posedness.

We make the following assumptions on the parameters in~\eqref{eq:dissipative-phs} and~\eqref{eq:boundary-conditions}:
\begin{enumerate}[label=(\arabic*)]
    \item\label{parameters:H} 
    $\calH \in L^{\infty}((0,1),\bbC^{n \times n})$ with $\calH(\zeta) ^{\ast} = \calH (\zeta) $ and $\calH (\zeta) \geq m I_n$ for a constant $m > 0$ and almost all $\zeta \in (0,1)$.%
    \footnote{
        Here, inequalities between self-adjoint matrices are meant in the sense of positive semi-definiteness.
    }

    \item\label{parameters:P} 
    $P_0, P_1 \in \bbC^{n \times n}$, where 
    $P_1$ is self-adjoint, and $P_0$ is skew-adjoint.

    \item\label{parameters:G} 
    $G_0, G_1 \in \bbC^{n \times r}$ for an integer $r \ge 1$.
	
    \item\label{parameters:S} 
    $S \in L^{\infty}((0,1); \bbC^{r \times r})$ and $\re S(\zeta) \coloneqq \frac{1}{2}\big(S(\zeta) + S(\zeta)^*\big) \ge \nu I_r$ for a constant $\nu > 0$ and almost all $\zeta > 0$.

    \item\label{parameters:V} 
    $\tilde W_\bdd \in \bbC^{(n+r) \times 2(n+r)}$ has full rank.
\end{enumerate}
Note that $P_0, P_1$ and $G_0,G_1$ do not depend on the spatial variable.
We endow the space $L^2\big((0,1);\bbC^n\big)$ with the inner product $\inner{\argument}{\argument}_{\calH}$ given by
\begin{equation*}
    \inner{f}{g}_{\calH} 
    \coloneqq 
    \frac{1}{2} \int_{0}^{1} g(\zeta)^{\ast} \calH(\zeta) f(\zeta) \dxInt \zeta
\end{equation*}
for all $f,g \in L^2\big((0,1);\bbC^n\big)$. 
Due to the assumptions on $\calH$ this inner product is equivalent to the usual inner product on $L^2\big((0,1);\bbC^n\big)$, but it is better adapted to our purposes.

Our main point of interest is that the second order derivative in~\eqref{eq:dissipative-phs} can cause internal damping.
If $S = 0$ (a case which we excluded in assumption~\ref{parameters:S}) and $P_1$ is invertible, then equation~\eqref{eq:dissipative-phs} defines a linear \emph{first order port-Hamiltonian system}.
These are well understood, see \cite{JaZw12,TrWa22,ViZwGoMa09}. 
In the next subsections we will demonstrate how one can derive properties of the second order differential equation~\eqref{eq:dissipative-phs} from the first order case. 

We formulate the evolution problem~\eqref{eq:dissipative-phs}--\eqref{eq:boundary-conditions} in operator theoretic language as the abstract Cauchy problem  
\begin{align*}
    \begin{cases}
        \dot z(t) = Lz(t) & \quad \text{for all } t \ge 0, \\
        z(0) = z_0, & 
    \end{cases}
\end{align*}
where the differential operator $L: \dom(L) \subseteq L^2\big((0,1); \bbC^n \big) \to L^2\big((0,1); \bbC^n \big)$ acts as 
\begin{align*}
    Lx 
    \coloneqq 
    \Big(
        P_1 \partial 
        \; + \; 
        P_0 
        \; + \; 
        \big(G_1 \partial + G_0\big) \, S \, \big(G_1^* \partial - G_0^* \big) 
    \Big) 
    \calH 
    \, x
\end{align*}
for all $x$ in 
\begin{align*}
    \dom(L) 
    \coloneqq 
    \Big\{
        x \in L^2\big((0,1); \bbC^n \big) 
        \suchthatManual 
        &
        \calH x \in H^1\big((0,1); \bbC^n\big) , \\ 
        &
        S \, \big(G_1^* \partial - G_0^* \big) \calH x \in H^1\big((0,1); \bbC^r \big) , \\
        & 
        x \text{ satisfies the boundary conditions~\eqref{eq:boundary-conditions}}
    \Big\}.
\end{align*}
The definition of $\dom(L)$ ensures that, if $z(t) \in \dom(L)$, then right hand side of~\eqref{eq:dissipative-phs} is in $L^2\big((0,1);\bbC^n\big)$ and the functions $\calH x(t)$ and $S \, \big(G_1^* \partial - G_0^* \big) \calH x(t)$ can be evaluated at the end points of the interval $(0,1)$ -- hence, the boundary conditions~\eqref{eq:boundary-conditions} make sense at the time $t$. 

In the next two subsections we discuss well-posedness and stability of the evolution equation in~\eqref{eq:dissipative-phs} and~\eqref{eq:boundary-conditions} -- 
i.e., we give conditions for $L$ to generate a contraction semigroup on $L^2\big((0,1); \bbC^n\big)$ with respect to the norm induced by the inner product $\inner{\argument}{\argument}_{\calH}$, and we study exponential stability of this semigroup.

\subsection{Representation via a closure relation}
\label{subsec:representation-via-cr}

We continue to use the notation from Subsection~\ref{subsec:diff-eq}.
We rewrite the differential operator $L$ in a form which fits the operator theoretic setting from Subsection~\ref{subsection:the-operators}: the goal is to represent $L$ as the operator $A_S$ (Proposition~\ref{prop:A_ext_and_A_S_phs}\ref{prop:A_ext_and_A_S_phs:itm:A_S}).
We follow essentially the approach from \cite[Chapter~6]{Vi07} where the generation property of $L$ has already been addressed. 
Our novel contribution is a result about stability,
Theorem~\ref{thm:stable-phs}.

We work on the Hilbert spaces 
\begin{align*}
    H_1 \coloneqq L^2\big((0,1); \bbC^n\big) 
    \qquad \text{and} \qquad 
    H_2 \coloneqq L^2\big((0,1); \bbC^r\big),
\end{align*}
where $H_1$ is endowed with the inner product $\inner{\argument}{\argument}_{\calH}$ and $H_2$ is endowed with the standard inner product. 
We denote elements of $H_1$ by $x$ and elements of $H_2$ by $x_p$.

\begin{definition} 
    \label{def:A_1_and_A_21_phs}
    On the spaces $H_1$ and $H_2$ we define operators $A_1$, $A_{21}$, and $S$ (as they appear in the abstract setting in Subsection~\ref{subsection:the-operators}) in the following way:
    \begin{enumerate}[label=(\alph*)]
        \item 
        Let $A_1: \dom(A_1) \subseteq H_1 \times H_2 \to H_1$ be given by 
        \begin{align*}
            \dom(A_1) \coloneqq &
            \bigg\{ 
                \begin{pmatrix}
                    x \\ x_p
                \end{pmatrix} 
                \in H_1 \times H_2 
                \suchthatManual 
                \begin{pmatrix}
                    \calH x \\ x_p
                \end{pmatrix} 
                \in H^1\big((0,1); \bbC^{n+r}\big) , \\ 
                & \qquad \qquad \qquad 
                \tilde W_\bdd 
                \begin{pmatrix}
                    \calH x(1) \\ 
                    x_p(1) \\ 
                    \calH x(0) \\ 
                    x_p(0)
                \end{pmatrix} 
                = 
                0
            \bigg\}
            \\ 
            A_1 
            \begin{pmatrix}
                x \\ x_p
            \end{pmatrix} 
            \coloneqq & 
            (P_1 \partial + P_0) \calH  x + (G_1 \partial + G_0) x_p
            .
        \end{align*}

        \item 
        Let $A_{21}: \dom(A_{21}) \subseteq H_1 \to H_2$ be given by 
        \begin{align*}
            \dom(A_{21}) 
            & \coloneqq 
            \left\{ x \in H_1 \suchthat \calH x \in H^1\big((0,1); \bbC^n\big) \right\}
            \\ 
            A_{21} x
            & = 
            (G_1^* \partial - G_0^*) \calH x
            .
        \end{align*}

        \item 
        The function $S \in L^\infty\big((0,1); \bbC^{r \times r}\big)$ from assumption~\ref{parameters:S} in Subsection~\ref{subsec:diff-eq} defines a bounded linear operator on $H_2$ by pointwise multiplication;  
        by abuse of notation, we also denote this operator by $S$. 
    \end{enumerate}
\end{definition}

Note that, due to the lower boundedness that was assumed in property~\ref{parameters:S} in Subsection~\ref{subsec:diff-eq}, the matrix multiplication operator $S$ on $H_2$ satisfies the coercivity assumption from Subsection~\ref{subsection:the-operators}. 
Next we compute $A_\ext$ and $A_S$.

\begin{proposition}
    \label{prop:A_ext_and_A_S_phs} 
    For the choices of $A_1$, $A_{21}$, and $S$ from Definition~\ref{def:A_1_and_A_21_phs}, 
    the operators $A_\ext$ and $A_S$ introduced in Definition~\ref{def:operators} take the following form:
    \begin{enumerate}[label=\upshape(\alph*)]
        \item\label{prop:A_ext_and_A_S_phs:itm:A_ext} 
        The operator $A_\ext: \dom(A_\ext) \subseteq H_1 \times H_2 \to H_1 \times H_2$ has the domain 
        \begin{align*}
            \dom(A_\ext) = \dom(A_1)
        \end{align*}
        and acts as
        \begin{align*}
            A_\ext = 
            \Biggl(
                \underbrace{
                    \begin{pmatrix}
                        P_1   & G_1 \\ 
                        G^*_1 & 0 
                    \end{pmatrix} 
                }_{\eqqcolon P_{1,\ext}}
                \partial 
                +
                \underbrace{
                    \begin{pmatrix}
                          P_0   & G_0 \\ 
                        - G^*_0 & 0
                    \end{pmatrix}
                }_{\eqqcolon P_{0, \ext}}
            \Biggl)
            \begin{pmatrix}
                \calH & 0 \\ 0 & \id_{H_2}
            \end{pmatrix}
            .
        \end{align*}

        \item\label{prop:A_ext_and_A_S_phs:itm:A_S} 
        The operator $A_S$ coincides with the differential operator $L$.
    \end{enumerate}
\end{proposition}

\begin{proof}
    This follows by a straightforward computation from the definitions of the involved operators.
\end{proof}

Note that assertion~\ref{prop:A_ext_and_A_S_phs:itm:A_S} in the proposition includes, in particular, that $A_S$ and $L$ have the same domain.

\subsection{Well-posedness and stability}
\label{subsec:well-posedness-and-stability}

Again, we use the notation from Subsections~\ref{subsec:diff-eq} and~\ref{subsec:representation-via-cr}.
Generation of a contraction semigroup by the operator $L$ has already been shown -- even under more general assumptions -- in \cite[Theorem~6.9]{Vi07}. 
For the convenience of the reader, we include both the result and its brief proof. 
If the self-adjoint matrix $P_{1,\ext} \in \bbC^{(n+r) \times (n+r)}$ introduced in Proposition~\ref{prop:A_ext_and_A_S_phs}\ref{prop:A_ext_and_A_S_phs:itm:A_ext} is invertible, we define
\begin{align*}
    W_\bdd 
    \coloneqq 
    \sqrt{2} \;
    \tilde W_B 
    \begin{pmatrix}
        P_{1,\ext} & -P_{1,\ext} \\ 
        I          &  I
    \end{pmatrix}%
    ^{-1} 
    \in 
    \bbC^{(n+r) \times 2(n+r)}
    .
\end{align*}
The invertibility of the $2(n+r) \times 2(n+r)$-matrix on the right follows from the invertibility of $P_{1,\ext}$, see e.g.\ \cite[Lemma~7.2.2]{JaZw12}. 
The factor $\sqrt{2}$ is not relevant for the formulation of the following results, 
but it plays a role in the interpretation of the matrix $W_\bdd$ by means of \emph{boundary flows} and \emph{efforts}, see \cite[Section~{7.2}]{JaZw12}.

\begin{proposition}
    \label{prop:well-posed-phs}
    Assume that the matrix $P_{1,\ext} \in \bbC^{(n+r) \times (n+r)}$ defined in Proposition~\ref{prop:A_ext_and_A_S_phs}\ref{prop:A_ext_and_A_S_phs:itm:A_ext} is invertible and that the $(n+r) \times (n+r)$-matrix 
    \begin{align*}
        W_\bdd  
        \begin{pmatrix}
            0      & I_{n+r} \\ 
            I_{n+r} & 0
        \end{pmatrix}
        W_\bdd^*
    \end{align*}
    is positive semi-definite. 
    Then $L$ generates a $C_0$-semigroup on $H_1 = L^2\big((0,1); \bbC^n\big)$ which is contractive with respect to the inner product $\inner{\argument}{\argument}_{\calH}$.
\end{proposition}

Note that this assumption that $P_{1,\ext}$ be invertible implies that $r \leq n$ (for if $r > n$, the columns of $G_1$ are linearly dependent, and hence the last $r$ columns of $P_{1, \ext}$ are linearly dependent as well).

\begin{proof}[Proof of Proposition~\ref{prop:well-posed-phs}]
    Invertibility of $P_{1,\ext}$ and the assumption on $W_\bdd$ imply that $A_\ext$ generates a contractive $C_0$-semigroup on $H_1 \times H_2$; 
    this follows from the theory of first order port-Hamiltonian systems \cite[Theorem 7.2.4]{JaZw12}. 
    Hence, the result from \cite[Theorem~2.2]{ZwGoMa16} that we quoted in Theorem~\ref{thm:inheritance-of-domination} implies that $A_S$ generates a contractive $C_0$-semigroup on $H_1$. 
    As $L = A_S$ according to Proposition~\ref{prop:A_ext_and_A_S_phs}\ref{prop:A_ext_and_A_S_phs:itm:A_S}, the claim follows.
\end{proof}

Finally, we give the sufficient condition for stability that we were aiming for.

\begin{theorem}
    \label{thm:stable-phs}
    Let the assumptions of Proposition~\ref{prop:well-posed-phs} be satisfied and assume in addition that the matrix $G_1 \in \bbC^{n \times r}$ has rank $n$. 
    If $(e^{tA_\ext})_{t \ge 0}$ is exponentially stable, then so is $(e^{t L})_{t \ge 0}$.
\end{theorem}

Note that the assumptions in the theorem imply that $n = r$, and thus that $G_1$ is invertible. 
Indeed, as pointed out after Proposition~\ref{prop:well-posed-phs} the assumption that $P_{1, \ext}$ is invertible (which is needed for the generation property of $L$) implies that $r \leq n$. 
But the assumption in Theorem~\ref{thm:stable-phs} that the $n \times r$-matrix $G_1$ has rank $n$ implies that, conversely, $n \le r$.

\begin{proof}[Proof of Theorem~\ref{thm:stable-phs}]
    We will use Theorem~\ref{thm:empty-peripheral-spec}\ref{thm:empty-peripheral-spec:itm:compact-embedding}. 
    To this end we only need to prove that the domain 
    $\dom(A_{21}) = \left\{ x \in X \suchthat \calH x \in H^1\big((0,1); \bbC^n\big) \right\}$, 
    endowed with the graph norm $\norm{\argument}_{A_{21}}$, embeds compactly into $(L^2([0,1]; \bbC^n),\norm{\argument}_{\calH})$.
    Let us show that this follows from the rank condition on $G_1$.
    Consider the following commutative diagram: 
    \begin{center}
        \begin{tikzpicture}
            \matrix(m)[matrix of math nodes,row sep=3em, column sep=4em, minimum width=2em]
            {
            (\dom(A_{21}), \norm{\argument}_{A_{21}}) & (L^2((0,1);\bbC^n), \norm{\cdot}_{\calH}) \\
            (H^1((0,1); \bbC^n), \norm{\argument}_{H^1}) & (L^2((0,1);\bbC^n), \norm{\argument}_{L^2}) \\
            };
            \path[-stealth]
            (m-1-1) edge node [left] {$x \mapsto \calH x$} (m-2-1)
                    edge node [above] {$j$} (m-1-2)
            (m-2-2) edge node [right] {$x \mapsto \calH^{-1} x$} (m-1-2)
            (m-2-1) edge node [below] {$i$} (m-2-2);
        \end{tikzpicture}
    \end{center}
    Here, $i$ and $j$ denote the canonical embeddings, and $i$ is well-known to be compact \cite[II.5.1]{Ev98}. 
    Moreover, the mapping $x \mapsto \calH^{-1} x$ on the right is clearly continuous. 

    Since, as observed before the proof, the rank condition on $G_1$ implies that $r = n$ and that $G_1$ is invertible, the adjoint matrix $G_1^*$ is invertible, too. 
    From this it follows by a brief computation that the mapping $x \mapsto \calH x$ on the left is also continuous. 
    Hence, $j$ is compact as claimed.
\end{proof}

\subsection{The heat equation on $(0,1)$ with non-local boundary conditions}
\label{subsec:heat-equation-nl}

As a very simple but illuminating toy example of the results discussed in Section~\ref{sec:phs}, let us consider the one-dimensional and scalar-valued heat equation with general boundary conditions.
As before, we denote the spatial derivative by $\partial$.

\begin{example}
    \label{exa:heat-equation}
    Let $S \in L^\infty\big((0,1);\bbR\big)$ satisfy $S(\zeta) \ge \nu$ for a constant $\nu > 0$ and almost all $\zeta \in (0,1)$. 
    We consider the heat equation 
    \begin{align*}
        \dot z(t) = \partial \big(S \partial z(t)\big)
        \quad \text{for } t \ge 0
    \end{align*}
    with unknown function $z: [0,\infty) \to L^2\big( (0,1); \bbC \big)$ where, for each $t \ge 0$, the function $x \coloneqq z(t)$ is subject to the boundary conditions
    \begin{align}
        \label{eq:heat-equation-bc-explicit}
        \tilde W_\bdd 
        \begin{pmatrix}
            x(1) \\ 
            S \partial x \, (1) \\ 
            x(0) \\ 
            S \partial x \, (0)
        \end{pmatrix}
        = 
        0;
    \end{align}
    here, $\tilde W_\bdd \in \bbC^{2 \times 4}$ is a fixed matrix of rank~$2$. 

    This fits the framework described in Subsection~\ref{subsec:diff-eq}: choose $n=r=1$, $P_1=P_0=0$, $G_0=0$, as well as $G_1 = 1$, and let $\calH$ denote the constant function on $(0,1)$ with value $1$. 
    The operator $L$ then acts as $L = \partial S \partial$ on the domain 
    \begin{align*}
        \dom(L) 
        =
        \big\{ 
            x \in H^1\big((0,1);\bbC\big) 
            \suchthatManual 
            &
            S \partial x \in H^1\big((0,1);\bbC\big) 
            \text{ and } \\ 
            &
            x \text{ satisfies the boundary conditions~\eqref{eq:heat-equation-bc-explicit}}
        \big\}
        .
    \end{align*}
    From Proposition~\ref{prop:well-posed-phs} one can get a condition for $L$ to generate a contraction semigroup: 
    let the matrices $\tilde W_{\bdd,1}, \tilde W_{\bdd,0} \in \bbC^{2 \times 2}$ denote the parts of $\tilde W_\bdd$ that belong to the boundary points $1$ and $0$, respectively, i.e., $\tilde W_\bdd = \big(\tilde W_{\bdd,1} \;\, \tilde W_{\bdd,0}\big)$.
    By a straightforward computation it then it follows from Proposition~\ref{prop:well-posed-phs} that $L$ generates a contraction semigroup in case that the $2 \times 2$-matrix
    \begin{align}
        \label{eq:heat-equation-pos-sd}
        \tilde W_{\bdd,1}
        \begin{pmatrix}
            0 & 1 \\ 
            1 & 0
        \end{pmatrix}
        \tilde W_{\bdd,1}^*
        \; - \;
        \tilde W_{\bdd,0}
        \begin{pmatrix}
            0 & 1 \\ 
            1 & 0
        \end{pmatrix}
        \tilde W_{\bdd,0}^*
    \end{align}
    is positive semi-definite. 

    Let us also explicitly write down the operator $A_\ext$ for this choice of parameters. 
    Proposition~\ref{prop:A_ext_and_A_S_phs}\ref{prop:A_ext_and_A_S_phs:itm:A_ext} shows that $A_\ext: \dom(A_\ext) \subseteq L^2\big((0,1);\bbC^2\big) \to L^2\big((0,1);\bbC^2\big)$ acts as
    $
        A_\ext 
        =
        \begin{psmallmatrix}
            0        & \partial \\ 
            \partial & 0 
        \end{psmallmatrix}
    $
    on
    \begin{align*}
        \dom(A_\ext) 
        = 
        \bigg\{
            \begin{pmatrix}
                x \\ x_p
            \end{pmatrix}
            \in 
            H^1\big((0,1);\bbC^2\big)
            \suchthatManual
            \tilde W_\bdd 
            \begin{pmatrix}
                x(1) \\ x_p(1) \\ x(0) \\ x_p(0)
            \end{pmatrix}
            = 0
        \bigg\}
        .
    \end{align*}
    It is worthwhile to point out that, while $A_\ext$ might look like a block operator at first glance, the splitting assumption~\ref{ass:splitting} is not satisfied, in general, unless for very specific choices of the matrix $\tilde W_\bdd$. 
    This is because the boundary conditions~\eqref{eq:heat-equation-bc-explicit} prevent $\dom(A_1)$ -- which coincides with $\dom(A_\ext)$ in the present situation, see Proposition~\ref{prop:A_ext_and_A_S_phs}\ref{prop:A_ext_and_A_S_phs:itm:A_ext} -- from splitting as a direct sum with components in $H_1$ and $H_2$.
    
    A sufficient condition for the exponential stability of the semigroups generated by $A_\ext$ and $L$ is given in the subsequent proposition.
\end{example}

\begin{proposition}
    \label{prop:heat-equation-stability}
    In the situation of Example~\ref{exa:heat-equation}, 
    assume that the $2\times 2$-matrix in~\eqref{eq:heat-equation-pos-sd} is positive semi-definite.
    Assume one of the following three conditions is satisfied: 
    \begin{enumerate}[label=\upshape(\arabic*)]
        \item\label{prop:heat-equation-stability:itm:1} 
        There exists a number $c > 0$ such that all vectors $b \in \ker \tilde W_B$ satisfy 
        \begin{align*}
            \re \big(b_1^* b_2 - b_3^* b_4\big) \le -c \big(\modulus{b_1}^2 + \modulus{b_2}^2\big)
            .
        \end{align*}
        
        \item\label{prop:heat-equation-stability:itm:0}
        There exists a number $c > 0$ such that all vectors $b \in \ker \tilde W_B$ satisfy 
        \begin{align*}
            \re \big(b_1^* b_2 - b_3^* b_4\big) \le -c \big(\modulus{b_3}^2 + \modulus{b_4}^2 \big).
        \end{align*}

        \item\label{prop:heat-equation-stability:itm:pos-def}
        The matrix in~\eqref{eq:heat-equation-pos-sd} is positive definite.
    \end{enumerate}
    Then the semigroup $(e^{tA_\ext})_{t \geq 0}$ on $L^2\big((0,1);\bbC^2\big)$ and the heat semigroup $(e^{tL})_{t \geq 0}$ on $L^2\big((0,1);\bbC\big)$ are exponentially stable.
\end{proposition}

\begin{proof}
    \ref{prop:heat-equation-stability:itm:1}
    According to \cite[Theorem~9.1.3]{JaZw12} exponential stability of $(e^{tA_\ext})_{t \geq 0}$ holds if there exists a number $k>0$ such that 
    \begin{align}
        \label{eq:inner-products-exp-stability}
        \inner{
            A_\ext 
            \begin{psmallmatrix}
                x \\ x_p
            \end{psmallmatrix}
        }{
            \begin{psmallmatrix}
                x \\ x_p
            \end{psmallmatrix}
        }%
        _{L^2\big((0,1);\bbC^2\big)}
        \; + \; 
        \inner{
            \begin{psmallmatrix}
                x \\ x_p
            \end{psmallmatrix}
        }{
            A_\ext 
            \begin{psmallmatrix}
                x \\ x_p
            \end{psmallmatrix}
        }%
        _{L^2\big((0,1);\bbC^2\big)}
    \end{align}
    is dominated by  
    $
        -k
        \norm{
            \begin{psmallmatrix} 
                x(1) \\ x_p(1)
            \end{psmallmatrix}
        }_{\bbC^2}^2
    $
    for every function
    $
       \begin{psmallmatrix} 
            x \\ x_p
        \end{psmallmatrix}
        \in \dom(A_\ext)
    $.
    By using the explicit form of $A_\ext$ given in Example~\ref{exa:heat-equation} one can check by integration of parts that the expression~\eqref{eq:inner-products-exp-stability} is equal to 
    \begin{align*}
        2 \re 
        \Big( 
            x(1)^* x_p(1) - x^*(0) x_p(0)
        \Big)
        .
    \end{align*}
    By assumption this number is -- as $\big( x(1), x_p(1), x(0), x_p(0) \big)^{\operatorname{T}} \in \ker \tilde W_B$ -- indeed dominated by $-2c \big(\modulus{x(1)}^2 + \modulus{x_p(1)}^2\big)$.
    This proves the claim for $k \coloneqq 2c$.
    
    As $(e^{tA\ext})_{t \geq 0}$ is exponentially stable, it follows from Theorem~\ref{thm:stable-phs} -- which is applicable since $G_1 = 1$ has rank $1$ -- that $(e^{tL})_{t \geq 0}$ is exponentially stable, too.

    \ref{prop:heat-equation-stability:itm:0} 
    The proof is the same as for~\ref{prop:heat-equation-stability:itm:1}; one just has to use the second condition in \cite[Theorem~9.1.3]{JaZw12} instead of the first one.

    \ref{prop:heat-equation-stability:itm:pos-def} 
    This is actually a special case of~\ref{prop:heat-equation-stability:itm:1} and~\ref{prop:heat-equation-stability:itm:0}, see \cite[Lemma~9.1.4]{JaZw12}.
\end{proof}

\begin{remark}
    The conditions in Proposition~\ref{prop:heat-equation-stability} do not characterize exponential stability of $(e^{tL})_{t \geq 0}$.
    In fact, in the most classical case of Dirichlet boundary conditions, where exponential stability of the heat semigroup is well-known, 
    we have
    \begin{align*}
        \tilde W_\bdd 
        =
        \begin{pmatrix}
            1 & 0 & 0 & 0 \\ 
            0 & 0 & 1 & 0
        \end{pmatrix}
        .
    \end{align*}
    One can readily check that none of the conditions~\ref{prop:heat-equation-stability:itm:1}, \ref{prop:heat-equation-stability:itm:0}, and~\ref{prop:heat-equation-stability:itm:pos-def} is satisfied in this case.

    However, the interesting point about Proposition~\ref{prop:heat-equation-stability} is that it yields exponential stability of the heat semigroup $(e^{tL})_{t \geq 0}$ for many examples of more involved boundary conditions. 
    To name just one arbitrary but explicit example, consider the matrix
    \begin{align*}
        \tilde W_\bdd 
        =
        \begin{pmatrix}
            0 & 1 & 1 & -1 \\ 
            1 & 1 & 0 &  0
        \end{pmatrix}
        ,
    \end{align*}
    which corresponds to the boundary conditions
    \begin{align*}
        \partial x(1) + x(0) + \partial x(0) & = 0, \\ 
        x(1) + \partial x(1)                 & = 0.
    \end{align*}
    A brief computation shows that, in this case, 
    the matrix in~\eqref{eq:heat-equation-pos-sd} is 
    $
        \begin{psmallmatrix}
            2 & 1 \\ 
            1 & 2
        \end{psmallmatrix}
        ,
    $
    which is positive definite. 
    Hence Proposition~\ref{prop:heat-equation-stability}\ref{prop:heat-equation-stability:itm:pos-def} implies that the heat semigroup with those boundary conditions is exponentially stable.
\end{remark}

We close the article with a comparison of Example~\ref{exa:heat-equation} to the approach via form methods, and with a brief outlook on potential further developments.

\begin{remark}
    It is illuminating to consider the special case of Robin boundary conditions and to compare the approach from Example~\ref{exa:heat-equation} to the approach via form methods that is possible in this case.
    Fix a matrix $C \in \bbC^{2 \times 2}$ and, again, a function $S \in L^\infty\big((0,1);\bbC\big)$. 
    We consider the second order differential operator $L = \partial S \partial$ with domain 
    \begin{align*}
        \dom(L) 
        \coloneqq
        \Big\{ 
            x \in H^1\big((0,1);\bbC\big) 
            \suchthatManual 
            &
            S \partial x \in H^1\big((0,1);\bbC\big) 
            \text{ and} 
            \\ 
            &
            \begin{pmatrix}
                \phantom{-} \partial S x(1) \\ 
                         -  \partial S x(0)
            \end{pmatrix}
            = 
            - C 
            \begin{pmatrix}
                x(1) \\ 
                x(0)
            \end{pmatrix}
        \Big\}
        .
    \end{align*}
    So the boundary conditions say that the exterior co-normal derivative of $x$ with respect to $S$ is equal to $-C$ times the boundary values of $x$. 
    These conditions are called \emph{Robin boundary conditions}. 
    If $C$ is a diagonal matrix, we speak of \emph{local} and otherwise of \emph{non-local} Robin boundary conditions. 
    
    We briefly discuss two approaches to study the operator $L$:
    \begin{enumerate}[label=(\alph*)]
        \item 
        One can consider $L$ as a special case of the operator in Example~\ref{exa:heat-equation} by setting 
        \begin{align*}
            \tilde W_\bdd 
            \coloneqq 
            \begin{pmatrix}
                c_{12} & 0 & c_{11} & -1 \\ 
                c_{22} & 1 & c_{21} &  0
            \end{pmatrix}
            ,
        \end{align*}
        where $c_{ij}$ are the components of $C$.
        Then a brief computation shows that the matrix in~\eqref{eq:heat-equation-pos-sd} is equal to $C + C^*$. 
        Hence, if $C + C^*$ is positive semi-definite, then $L$ generates a contractive $C_0$-semigroup.

        \item 
        One can also construct $L$ via the theory of sesquilinear forms that we briefly outline at the beginning of Subsection~\ref{subsec:forms}.
        Indeed, it is not difficult to check that $-L$ is associated to the form $a: H^1\big((0,1); \bbC\big) \times H^1\big((0,1); \bbC\big) \to \bbC$ on $L^2\big((0,1);\bbC\big)$ that is given by
        \begin{align*}
            a(w,x) 
            \coloneqq 
            \int_0^1 \partial w(\zeta)^* \; S \; \partial x(\zeta) \dxInt \zeta 
            + 
            \inner{
                \begin{psmallmatrix}
                    w(0) \\ 
                    w(1)
                \end{psmallmatrix}%
            }{
                C
                \begin{psmallmatrix}
                    x(0) \\ 
                    x(1)
                \end{psmallmatrix}
            }_{\bbC^2}
        \end{align*}
        for all $w,x \in H^1\big((0,1); \bbC\big)$.
        If one endows $H^1\big((0,1); \bbC\big)$ with it usual norm, then $a$ can easily be checked to be a closed form in the sense specified in Subsection~\ref{subsec:forms}.
        Moreover it is not difficult to check that the operator associated to $a$ is $-L$. 
        Hence, $L$ generates an analytic $C_0$-semigroup.
        If $C+C^*$ is positive semi-definite, than the numerical range of the form is contained in the closed right half plane in $\bbC$ and hence the semigroup is contractive in this case. 
    \end{enumerate}
    So in the particular case of Robin boundary conditions we can use both approaches; they give the same sufficient condition of contractivity of the semigroup generated by $L$. 
    The form method, however, also directly yields analyticity. 

    Interestingly, the construction of the Robin Laplace operator via form methods does not fit the framework of our Theorem~\ref{thm:via-forms} 
    as we only consider bounded operators $A_{11}$ in this theorem.
\end{remark}

We note that especially the case of local Robin boundary conditions has been studied in great generality on multi-dimensional domains; see for instance \cite{Da00}. 

Non-local Robin boundary conditions occur, for instance, in a physical model in \cite[Section~3]{GuMe97}.
At a theoretical level, they provide in interesting source of examples for so-called \emph{eventually positive} behaviour of $C_0$-semigroups, see e.g.\  \cite[pp.\,2625--2630]{DKG16b}. 
As eventual positivity of a semigroup is closely tied to properties of the eigenvalues with maximal real part \cite[Theorem~4.4 and Corollary~3.3]{DKG16b}, it is an interesting question whether the characterization of imaginary eigenvalues of $A_S$ in Theorem~\ref{thm:eigenvalues} can be used to obtain further examples of non-local boundary conditions -- beyond the Robin case -- for which the heat equation shows eventually positive behaviour.

\subsection*{Conflict of interest statement}

The authors declare that they do not have any conflict of interest.

\subsection*{Data availability statement}

The manuscript does not use or generate any data sets.

\bibliographystyle{abbrv}
\bibliography{literature}

\end{document}